\documentclass{siamltex}
\usepackage{amsfonts,amsmath}
\usepackage{amssymb}
\usepackage{hyperref}
\usepackage{graphicx}
\usepackage{wrapfig}
\usepackage{subfig}
\usepackage{epsfig}
\usepackage{float}
\usepackage{bbm}

\newtheorem{algorithm}{Algorithm}
\usepackage{algpseudocode}
\usepackage[font=scriptsize]{caption}
\usepackage{chngcntr}
\counterwithout{figure}{section}
\counterwithout{table}{section}
\usepackage{color}
\title{Computation of extremal eigenvalues of high-dimensional lattice-theoretic tensors via tensor-train decompositions}

\author{H.~Hakula\footnotemark[2]
\and P.~Ilmonen\footnotemark[2]
\and V.~Kaarnioja\footnotemark[2]~\footnotemark[3]
}

\begin{document}
\maketitle

\renewcommand{\thefootnote}{\fnsymbol{footnote}}

\footnotetext[2]{Aalto University, Department of Mathematics and Systems Analysis, P.O. Box 11100, FI-00076 Aalto, Finland (harri.hakula@aalto.fi, pauliina.ilmonen@aalto.fi, vesa.kaarnioja@helsinki.fi). %
The work of HH was supported by the European Union's Seventh 
Framework Programme (FP7/2007--2013) ERC grant agreement no 339380. The work of VK was supported by the Academy of Finland (decision 267789).}
\footnotetext[3]{Department of Mathematics and Statistics, P.O. Box~68 (Gustaf H\"{a}llstr\"{o}min katu 2B), FI-00014 University of Helsinki, Finland.}

\begin{abstract}
This paper lies in the intersection of several fields: number theory, lattice theory, multilinear algebra, and scientific computing. We adapt existing solution algorithms for tensor eigenvalue problems to the tensor-train framework. As an application, we consider eigenvalue problems associated with a class of lattice-theoretic meet and join tensors, which may be regarded as multidimensional extensions of the classically studied meet and join matrices such as GCD and LCM matrices, respectively. In order to effectively apply the solution algorithms, we show that meet tensors have an explicit low-rank tensor-train decomposition with sparse tensor-train cores with respect to the dimension. Moreover, this representation is independent of tensor order, which eliminates the so-called curse of dimensionality from the numerical analysis of these objects and makes the solution of tensor eigenvalue problems tractable with increasing dimensionality and order. For LCM tensors it is shown that a tensor-train decomposition with an a priori known maximal TT rank exists under certain assumptions. We present a series of easily reproducible numerical examples covering tensor eigenvalue and generalized eigenvalue problems that serve as future benchmarks. The numerical results are used to assess the sharpness of existing theoretical estimates.
\end{abstract}

\renewcommand{\thefootnote}{\arabic{footnote}}



\pagestyle{myheadings}
\thispagestyle{plain}
\markboth{H. HAKULA, P. ILMONEN, AND V. KAARNIOJA}{COMPUTING EIGENVALUES OF LATTICE-THEORETIC TENSORS}

\section{Introduction}
\label{intro}
This paper lies in the intersection of several fields: number theory, lattice 
theory, multilinear algebra, and scientific computing.

Tensors (also called hypermatrices or multidimensional arrays by other authors) 
are a multivariate generalization of matrices and they appear frequently in a variety 
of applications spanning different fields of mathematics. In the numerical analysis 
of tensors, an immediate concern that arises is the so-called \emph{curse of dimensionality}. 
When considered as an array, the number of elements that must be stored increases 
exponentially with respect to the order of the tensor.

Tensors arising from applications tend to possess some underlying structure, which can be used to overcome the associated computational
limitations. This hidden structure can be uncovered by using tensor decompositions, such as the canonical decomposition and the Tucker decomposition, both of which were introduced already by Hitchcock in 1927~\cite{hitchcock}. These formats can be used to reduce the storage and computational complexity of tensors. For instance, the canonical decomposition allows the storage of
certain classes of tensors in a number of parameters that scale linearly with the tensor order. This comes with the cost of an
additional parameter called the rank of a tensor. The concept of the
tensor rank is analogous to that of matrices, i.e., a rank-$r$ tensor
can be represented by $r$ rank-1 tensors within the tensor
decomposition paradigm. Unlike the case of matrices, it may happen 
that the optimal rank $r$ is much higher than the
dimension of the tensor~\cite{Krus89} with its value contingent on
the type of decomposition that is being used. The major setbacks of
these methods are based on the difficulty of forming the
decomposition: it either has to be known analytically or an
approximation must be sought using optimization-based methods.

The previous decade saw the introduction of a new generation of
tensor decompositions: the hierarchical Tucker
format~\cite{Grasedyck10,HackKuhn09} and the tensor-train
decomposition~\cite{Osel11,breaking}, where the latter may be regarded as a
special case of the former~\cite{Hackbusch14}. These formats have
been designed so that a representative can be found for \emph{any}
tensor in a stable, algorithmic series of operations within
prescribed accuracy. Moreover, both formats allow interpolation
within the decomposition framework. This means that if, say, each
element of the tensor can be computed by an evaluation of a function
$f\!:\mathbb{R}^{n\times\cdots\times n}\to\mathbb{R}$, then the
tensor decomposition may be formed without ever explicitly constructing the full
$n\times\cdots\times n$ array of values, thus bypassing the curse of
dimensionality. While the precise convergence conditions of these
interpolation methods still remain largely unknown, this methodology has
been applied successfully to many problems of interest such as
high-dimensional integration~\cite{OselTyr10}, solution of stochastic
partial differential equations~\cite{Dolgov15}, and extremal
eigenvalues of tensor products of matrices~\cite{Dolgov14}. We refer to the comprehensive survey~\cite{GKT13} for a detailed overview and applications of low-rank tensor approximation techniques.

The eigenvalues of tensors were defined independently by Lim~\cite{Lim} and Qi~\cite{Qi} in 2005. There has been a surge of activity in the tensor
eigenvalue solution techniques since then, see for instance~\cite{KoMa10} and \cite{KoMa14},
and references therein. The dominant real eigenpair of a symmetric tensor can be used to construct its best symmetric rank-1 approximation~\cite{KofReg02,Lathauwer00}. On the other hand, the minimal real eigenvalue of a symmetric tensor is an indicator of the positive definiteness of a symmetric tensor. Tensor eigenvalue solution techniques can be used to analyze these important characteristics, which have immediate applications in diffusion tensor imaging~\cite{dti3,dti2,dti4}, automatic control~\cite{Qi,ac}, image authenticity verification~\cite{iav}, the multilinear PageRank algorithm~\cite{pagerank}, and so on. The limiting probability distributions of higher order Markov chains can be determined by solving the dominant eigenvectors of transition probability tensors, and properties of limiting probability distributions are intertwined with the spectral theory of nonnegative tensors~\cite{markov,Ng09}. Tensor eigenvalue problems have
emerged in other contexts as well, such as analysis of Newton's method~\cite{NLA:NLA1883},
and this trend is likely to continue as the solution techniques are developed further. We refer to the monograph~\cite{qibook} for further details on the spectral theory of tensors.

In this work, we have adapted the symmetric high-order power method (S-HOPM) and generalized eigenproblem adaptive power method (GEAP) for use with TT tensors, which drastically lessens the computational complexity of these algorithms --- and which may be useful in future applications beyond the lattice-theoretic tensors considered in this paper. To accommodate these algorithms, we consider in this treatise the eigenvalue problems associated with lattice-theoretic meet and join tensors, specifically GCD and LCM tensors. As the foundation of the numerical tests, an explicit low-rank tensor-train decomposition for meet tensors that is completely independent of the tensor order $d$~is derived. The TT cores comprising this representation are sparse with respect to the dimension $n$, and in the case of the so-called Smith tensor, the number of nonzero elements that must be stored is bounded by $\mathcal{O}(n\ln n)$. In particular, this ensures that numerical experiments can be carried out in very high dimensions. The decomposition theory regarding join tensors is still undeveloped, and so we construct the TT decomposition for LCM tensors numerically using the TT-DMRG cross algorithm. Moreover, we prove that a tensor-train decomposition exists under certain assumptions with an a priori known maximal TT rank. This implies, in consequence, that the number of parameters required to represent LCM tensors in the TT format is bounded by a term depending only moderately on the tensor order. Using the TT decompositions of GCD and LCM tensors in conjunction with the eigenvalue algorithms modified for use with the TT format, we present a series of easily reproducible numerical examples covering tensor eigenvalue and generalized eigenvalue problems that serve as future benchmarks. The numerical results are used to assess the sharpness of the existing theoretical estimates. We note that GCD and LCM tensors provide excellent benchmarks since these tensors form a large group of non-sparse and non-random tensors, which possess easily constructible low-rank and --- in the case of GCD tensors --- sparse tensor decompositions.

\subsection{Outline of the paper}
The rest of this paper is organized as follows:
First the notation and basic definitions regarding tensor eigenvalues and the tensor-train decomposition are given in Section~\ref{notations}.
We derive the implementations of numerical eigenvalue algorithms in the tensor-train formalism in Section~\ref{TTalgorithms}. In Section~\ref{meetandjointensors}, we move onto the properties of lattice-theoretic meet and join tensors. Subsection~\ref{meetTT} contains the derivation of a
tensor-train representation for meet tensors and Subsection~\ref{joinTT} considers the implicit representation of join tensors.
The numerical experiments are presented in Section~\ref{numex}. Finally, discussion on results and future prospects
concludes the paper.
\section{Notations and preliminaries}
\label{notations}
\subsection{Symmetric tensors and tensor eigenvalue problems}
An $n$-dimen-sional tensor $A$ of order $d$ (over a scalar field) is defined as an array of $n^d$ elements
\[
A_{i_1,\ldots,i_d}
\]
for $i_1,\ldots,i_d\in\{1,\ldots,n\}$. The tensor $A$ is said to be symmetric (see~\cite{Comon08}) if its entries are invariant over any permutation of its indices.

Let $x=(x_1,\ldots,x_n)^\top$. The mode-$k$ contraction of the tensor $A$ against the vector $x$ is defined as a new tensor $B=A\times_k x$, the elements of which are given by
\[
B_{i_1,\ldots,i_{k-1},1,i_{k+1},\ldots,i_d}=\sum_{i_k=1}^nA_{i_1,\ldots,i_k,\ldots,i_d}x_{i_k}
\]
for $i_1,\ldots,i_{k-1},i_{k+1},\ldots,i_d\in\{1,\ldots,n\}$. A symmetric $n$-dimensional real tensor $A$ of order $d$ uniquely defines a homogeneous (real) polynomial $p(x)=Ax^d$ of degree $d$ in $n$ variables, where
\[
Ax^d=A\times_1x\times_2\cdots\times_dx.
\]
A tensor $A$ is called positive definite, if the values of this polynomial satisfy
\[
p(x)>0\quad\text{for all }~x\in\mathbb{R}^n\setminus\{0\}.
\]
It follows that a symmetric tensor $A$ can be positive definite only when $d$ is even.

Similarly to the above, let $Ax^{d-1}$ denote a vector in $\mathbb{R}^n$ and $Ax^{d-2}$ a matrix in $\mathbb{R}^{n\times n}$ defined by setting
\[
Ax^{d-1}=A\times_2x\cdots\times_dx\quad\text{and}\quad Ax^{d-2}=A\times_3x\cdots\times_dx,
\]
respectively. Moreover, let $x^{[\alpha]}$ denote a vector in $\mathbb{R}^n$ given by
\[
x^{[\alpha]}=(x_1^\alpha,\ldots,x_n^\alpha)^\top
\]
for $\alpha\in\mathbb{R_+}$.

The eigenvalue problem for tensors as considered in this paper
was first introduced by Lim~\cite{Lim} and Qi~\cite{Qi}. It may be regarded as a special case of the following generalized tensor eigenvalue problem discussed by Chang, Pearson, and Zhang~\cite{ChaPeaZha09}, and both have been considered by several authors since~\cite{CuiDaiNie14,KoMa14}.
\begin{definition}[Generalized eigenvalue problem]
Let both $A$ and $B$ be $n$-dimensional tensors of order $d$. Then $\lambda\in\mathbb{R}$ is called a $B$-eigenvalue if
\begin{equation}
Ax^{d-1}=\lambda Bx^{d-1},\label{Bevp}
\end{equation}
where $x\in\mathbb{R}^n\setminus\{0\}$ is the $B$-eigenvector.
\end{definition}

The generalized eigenvalue problem encapsulates several classes of tensor eigenvalue problems of interest.
\begin{definition}[H-eigenvalue problem]
Let $A$ be an an $n$-dimensional tensor of order $d$. The number $\lambda\in\mathbb{R}$ and the vector $x\in\mathbb{R}^n\setminus\{0\}$ are called the H-eigenvalue and H-eigenvector if
\begin{equation}
Ax^{d-1}=\lambda x^{[d-1]}.\label{Hevp}
\end{equation}
This problem is equivalent to solving the problem~\eqref{Bevp} with $B$ given by $B_{i_1,\ldots,i_d}=\delta_{i_1,\ldots,i_d}$, where $\delta$ represents the Kronecker tensor defined as unity whenever its indices coincide and vanishing otherwise.
\end{definition}
\begin{definition}[Z-eigenvalue problem]
Let $A$ be an $n$-dimensional tensor of order $d$. The number $\lambda\in\mathbb{R}$ and the vector $x\in\mathbb{R}^n\setminus\{0\}$ are called the Z-eigenvalue and Z-eigenvector if
\begin{equation}
Ax^{d-1}=\lambda x\quad\text{and}\quad\|x\|=1.\label{Zevp}
\end{equation}
For an even order $d$, this problem is equivalent to solving the problem~\eqref{Bevp} with $B=\mathcal{E}$, the identity tensor such that $\mathcal{E}x^{d-1}=\|x\|^{d-2}x$ for all $x\in\mathbb{R}^n$.
\end{definition}

\emph{Remark.} The complex-valued solutions $\lambda\in\mathbb{C}$ and $x\in\mathbb{C}^{n}\setminus\{0\}$ to~\eqref{Hevp} are simply called eigenvalues and eigenvectors, and the complex solutions to~\eqref{Zevp} are called E-eigenvalues and E-eigenvectors~\cite{Qi}. However, the available solution methods for complex eigenvalues are practically limited to solving systems of polynomial equations in several variables directly, which can be carried out by using, e.g., the polynomial system solver \texttt{NSolve} available in Mathematica. This approach is not sustainable for problems with high dimensionality (order) since it cannot take advantage of the structure of the tensor. Mathematica's \texttt{NSolve} function for polynomial systems utilizes the notoriously computationally expensive solution of Gr\"{o}bner bases, which requires access to all $n^d$ elements of the tensor. While the possibility of an extension of the shifted power method to a certain class of complex tensor eigenvalues has been discussed in~\cite{KoMa10}, it is not covered by the convergence theory established for real-valued shifted power method for tensors and hence it is omitted from this paper. Moreover, solving the whole spectrum of a tensor quickly becomes intractable with increasing dimensionality since the number of eigenvalues increases exponentially with respect to tensor order~\cite{Qi}.

For the reasons discussed above, we restrict our study to real-valued dominant and minimal eigenvalues of even order tensors in this work. Even order ensures that both H- and Z-eigenvalues exist with real solutions to the corresponding eigenvalue-eigenvector equations. The minimal H- and Z-eigenvalues of tensors are particularly important in the highly nontrivial problem of identifying positive definite tensors: It was proven by Qi~\cite{Qi} that an even order, symmetric tensor $A$ is positive definite, if and only, if all of its H-eigenvalues (resp.~Z-eigenvalues) are positive. On the other hand, dominant Z-eigenvalues are closely related to the rank-1 approximation of tensors studied by De Lathauwer et al.~\cite{Lathauwer00} and by Kofidis and Regalia~\cite{KofReg02}. In the sequel, we propose efficient implementations of the power method and adaptive shifted power method for the solution of extremal eigenvalues in the formalism of the tensor-train decomposition in Section~\ref{TTalgorithms}.

The generalized eigenvalue problem may also be expressed succintly as a constrained optimization problem: $(\lambda,x)\in\mathbb{R}\times(\mathbb{R}^n\setminus\{0\})$ is a generalized $B$-eigenpair, if and only if $x$ is critical point and $\lambda$ is a critical value of the problem
\begin{align}
\max Ax^d\quad \text{such that}\quad Bx^d=1.\label{opt}
\end{align}
This restatement of the problem follows immediately by noticing that $\nabla_x(Ax^d)=dAx^{d-1}$ and $\nabla_x(Bx^d)=dBx^{d-1}$, from which it is apparent that the $B$-eigenvalue problem~\eqref{Bevp} is precisely the Lagrange multipliers problem $\nabla_{x,\lambda}\mathcal{L}(x,\lambda)=0$ for the Lagrangian $\mathcal{L}(x,\lambda)=Ax^d-\lambda Bx^d$.

\subsection{Tensor-train decomposition}
\label{ttdecomposition}
The tensor-train decomposition was introduced in~\cite{breaking} as a generalization of the classical rank-1 canonical decomposition of tensors. In his paper~\cite{Osel11}, Oseledets presented an algorithm to approximate a given $n\times\cdots\times n$ tensor $A$ by another tensor $B\approx A$ with elements
\begin{equation*}
B_{i_1,\ldots,i_d}=G_1(i_1)\cdots G_d(i_d),\quad i_1,\ldots,i_d\in\{1,\ldots,n\},
\end{equation*}
where each $G_k(i_k)$ is an $r_{k-1}\times r_k$ matrix and $r_0=r_d=1$. The three-dimensional tensors $G_k$ are called the \emph{TT cores} of the TT decomposition and $r_k$ the \emph{compression ranks}. For general tensors $A$, it was shown by Oseledets that an approximation $B$ in this form can be found in the neighborhood of $A$ within arbitrarily prescribed tolerance $\varepsilon>0$ in the Frobenius norm, viz. $\|A-B\|_F<\varepsilon\|A\|_F$. For notational convenience, we define the transpose of third order TT cores $G_k$, $k\in\{2,\ldots,d-1\}$, as the permutation of their first two modes, i.e.,
\[
(G_k^\top(i_1))_{i_2,i_3}=G_k(i_2)_{i_1,i_3},
\]
where $i_1\in\{1,\ldots,r_{k-1}\}$, $i_2\in\{1,\ldots,n\}$, and $i_3\in\{1,\ldots,r_k\}$.

The compression ranks are bounded below by the ranks of the unfolding matrices $A_{[k]}$ defined elementwise by setting
\begin{equation*}
(A_{[k]})_{(i_1,\ldots,i_k),(i_{k+1},\ldots,i_d)}=A_{i_1,\ldots,i_d},
\end{equation*}
where the multi-indices $(i_1,\ldots,i_k)\in\{1,\ldots,n\}^k$ enumerate the rows and the multi-indices $(i_{k+1},\ldots,i_d)\in\{1,\ldots,n\}^{d-k}$ enumerate the columns of the $n^k\times n^{d-k}$ matrix $A_{[k]}$. Moreover, these compression ranks are also achievable.
\begin{theorem}[cf.~\cite{Osel11}]
There exists a TT decomposition of the $n$-dimensional order $d$ tensor $A$ with compression ranks
\[
r_k={\rm rank}\,A_{[k]},\quad k=1,\ldots,d-1.
\]
\end{theorem}

This motivates the definition of the maximal \emph{tensor-train rank} of tensor $A$ by setting
\[
{\rm rank}_{\rm TT}A=\max_{1\leq k\leq d-1}{\rm rank}\, A_{[k]}.
\]

The TT decomposition possesses many desirable properties in view of mitigating the curse of dimensionality, such as the storage of the tensor $B$ in $(d-2)nr^2+2dr$ parameters. Furthermore, the complexity of many important operations in linear algebra scales linearly with respect to $d$ in the TT format. In this formalism, the mode-$k$ contraction of tensor $B$ --- now characterized completely by its constituent TT cores $(G_k)_{k=1}^d$ --- against a vector $x\in\mathbb{R}^n$ can be expressed as
\[
(B\times_k x)_{i_1,\ldots, i_{k-1},1,i_{k+1},\ldots,i_d}=G_1(i_1)\cdots G_{k-1}(i_{k-1})(x^\top G_k^\top)G_{k+1}(i_{k+1})\cdots G_d(i_d),
\] 
where the vector-tensor product is understood as a vector-matrix product over each matrix $G_k^\top(\cdot)$ such that
\[
x^\top G_k^\top=\begin{pmatrix}x^\top G_k^\top (1)\\ \vdots\\ x^\top G_k^\top(r_{k-1})\end{pmatrix}.
\]

If the elements of the tensor $A$ can be represented by a functional $f\!:\mathbb{R}^{n\times\cdots\times n}\to\mathbb{R}$ such that $A_{i_1,\ldots,i_d}=f(i_1,\ldots,i_d)$, then the TT decomposition can be formed implicitly by using cross-interpolation. An effective implementation is based on the notions of the density matrix renormalization group (DMRG)~\cite{White93} and the matrix CUR decomposition~\cite{GorZamTyr97}. Like the ALS method used to obtain an approximate canonical decomposition, the precise convergence criteria of the DMRG scheme remain unknown at the time of writing. An additional advantage of the DMRG method besides cross-interpolation is that it is rank-revealing and thus yields a compressed TT representation for tensors with low rank. In the following, we give a brief overview of the TT-DMRG cross algorithm. This scheme has been covered before in the literature and we recommend the interested readers to see~\cite{SavOsel11} for a detailed description.

The TT-DMRG cross algorithm attempts to find a TT tensor $B$ with TT cores $(G_k)_{k=1}^d$ such that
\begin{equation}
\min_{G_1,\ldots,G_d}\|A-B\|_F.\label{als}
\end{equation}
In the DMRG scheme, the optimization is performed over two TT cores $G_k$ and $G_{k+1}$ at a time in successive sweeps starting from $k=1$ and proceeding to $k=d-1$ and then from $k=d-1$ back to $k=1$. At the $k^\text{th}$ sweep, the structure of the TT format transforms the problem of finding the supercore $W_k(i_k,i_{k+1})=G_k(i_k)G_{k+1}(i_{k+1})$ solving~\eqref{als} into a quadratic optimization problem, which is equivalent to solving a small linear system. The TT cores $G_k$ and $G_{k+1}$ can be recovered from the SVD of $W_k$ producing a new approximation with an updated compression rank $r_k$. The next relevant core $W_k$~is determined by using the maximum volume principle on the most important modes $A_{i_1,\ldots,i_{k-1},\cdot,\cdot,i_{k+2},\ldots,i_d}$. The maximum volume principle can be implemented numerically using the computationally inexpensive MaxVol algorithm~\cite{GorOse10}.
\section{Algorithms for tensor eigenvalue problems in the TT format}\label{TTalgorithms}
In recent years, several schemes have been proposed to solve the tensor eigenvalue problems \eqref{Bevp}--\eqref{Zevp}. Ng, Qi, and Zhou~\cite{Ng09} studied the generalization of the classical matrix power method to determine the dominant H-eigenpairs of symmetric tensors. A power method for tensors was studied by De Lathauwer et al.~\cite{Lathauwer00} and by Kofidis and Regalia~\cite{KofReg02} in the context of optimal rank-1 approximations of tensors, which has since been adapted to solving Z-eigenvalues. Kolda and Mayo proposed a shifted power method to determine all $B$-eigenpairs \cite{KoMa14}, which was developed concurrently with another method to solve all $B$-eigenpairs in the polynomial optimization framework \cite{CuiDaiNie14}. However, the aforementioned solution schemes are subject to the curse of dimensionality when the tensor $A$ is considered as a na\"{i}ve $n\times\cdots\times n$~array since the evaluation of the product $Ax^d$ requires $n^d$ evaluations. By taking the symmetry into account, the storage cost of an $n$-dimensional order $d$ tensor is of order $\mathcal{O}(n^d/d!)$.

When the array is given in the TT format, the evaluation of the products $Ax^d$, $Ax^{d-1}$, and $Ax^{d-2}$ scales linearly with respect to $d$ which significantly improves the performance of the existing algorithms. In this section, we describe the efficient implementation of these algorithms for tensors given in the TT format. While the power method and shifted power method for tensor eigenvalues have been covered in the literature before (see the references above), their efficient application for tensors given in the TT format has not been covered in the existing literature and is presented in the following subsections. The algorithms can be applied to \emph{any} tensor given in the TT format --- and since a TT decomposition of an arbitrary tensor can be determined within numerical precision using, e.g., the TT-DMRG cross scheme, the algorithms presented in the following are applicable for more general classes of tensors than the lattice-theoretic tensors considered in this paper. Naturally, the algorithms are efficient provided that the TT decomposition is either low-rank or the TT cores are sufficiently sparse.
\subsection{Symmetric high-order power method in the TT format}
The largest eigenvalues of a symmetric $n$-dimensional order $d$ tensor $A$ can be determined by using a power method analogous to that of matrices~\cite{KofReg02,Ng09}. For any initial guess $x_0\in\mathbb{R}^n$, $\|x_0\|=1$, the \emph{symmetric high-order power method} (S-HOPM) is defined by the successive iterates
\[
y_k=Ax_{k-1}^{d-1},\quad x_{k}=\frac{y_{k-1}^{\left[\frac{1}{m-1}\right]}}{\big\|y_{k-1}^{\left[\frac{1}{m-1}\right]}\big\|},\quad\text{and}\quad\lambda_{k}=\frac{Ax_{k}^d}{\|x\|_m^m},\quad k=1,2,\ldots,
\]
where $\|\!\cdot\!\|_p$ denotes the $p$-norm. Setting $m=d$ corresponds to H-eigenvalues and $m=2$ corresponds to Z-eigenvalues.

The S-HOPM uses only tensor-vector contractions to operate with the tensor. When $A$ is given as a TT tensor, the power method can be implemented efficiently in the following way.
\begin{algorithm}[TT-S-HOPM]\label{ttshopm}
Let $(G_k)_{k=1}^d$ be the TT cores of a tensor $A$ and $x_0\in\mathbb{R}^n$ an initial guess such that $\|x_0\|=1$. For H-eigenvalues, set $m=d$. Otherwise, set $m=2$~for Z-eigenvalues. Then iterate
\begin{algorithmic}
\For{$k=1,2,\ldots$} 
\State $y_k=Ax_{k-1}^{d-1}=G_1(x_{k-1}^\top G_2^\top)\cdots(x_{k-1}^\top G_{d-1}^\top)(x_{k-1}^\top G_d^\top)$
\State $z_k=y_k^{\left[\frac{1}{m-1}\right]}$
\State $x_k=z_k/\|z_k\|$
\State $\lambda_k=Ax_k^d/\|x_k\|_m^m=(x_k^\top G_1^\top)(x_k^\top G_2^\top)\cdots(x_k^\top G_{d-1}^\top)(x_k^\top G_d^\top)/\|x_k\|_m^m$
\EndFor
\end{algorithmic}
until convergence.
\end{algorithm}

The convergence properties of S-HOPM for both H-eigenvalues and Z-eigenvalues have been studied in recent literature, and the following facts are known.
\begin{itemize}
\item For H-eigenvalues, the algorithm is known to converge to the dominant H-eigenpair $(\lambda,x)\in\mathbb{R}\times(\mathbb{R}^n\setminus\{0\})$ with $\lambda_k\to\lambda$ and $x_k\to x$ for any initial guess $x_0\in\mathbb{R}_+^n$ if $A$ belongs to the class of primitive tensors~\cite{ChaPeaZha11} (this is the case, e.g., if $A$ only has positive entries). More general well-conditioned classes of tensors have since been identified~\cite{FGH13}.
\item For Z-eigenvalues, the algorithm converges to a local maximum of~\eqref{opt} with the identity tensor $B=\mathcal{E}$ if $f(x)=Ax^d$ is a convex function for all $x\in\mathbb{R}^n$ and $d$ is even~\cite{KofReg02}. However, this ensures only the convergence of the sequence $(\lambda_k)_{k=1}^\infty$ toward a Z-eigenvalue. The sequence $(x_k)_{k=1}^\infty$ need not converge toward the corresponding Z-eigenvector.
\end{itemize}

Since the convergence of the S-HOPM method with $m=2$ is only ensured toward local maxima of~\eqref{opt}, multiple initial guesses may be necessary to find the dominant Z-eigenvalue.

\subsection{Generalized eigenproblem adaptive power method in the TT format}
Let $A$~and $B$ be $n$-dimensional, symmetric tensors of order $d$. Additionally, let $B$~be positive definite. The algorithm suggested by Kolda and Mayo~\cite{KoMa14} is based on the formulation of the $B$-eigenvalue problem~\eqref{Bevp} as an optimization problem
\begin{equation}
\max f(x)=\frac{Ax^d}{Bx^d}\|x\|^d\quad\text{subject to}\quad x\in\Sigma,\label{geapopt}
\end{equation}
where $\Sigma=\{x\in\mathbb{R}^n\mid \|x\|^d=1\}$. If $f$ is locally convex around the trial point $x\in\Sigma$, then performing the update $x\leftarrow\nabla f(x)/\|\nabla f(x)\|$ will yield ascent. Conversely if $f$ is locally concave, then the update $x\leftarrow-\nabla f(x)/\|\nabla f(x)\|$ yields descent. However, the function $f$~is generally neither convex nor concave and to remedy this situation, it is preferable to work instead with a shifted function
\[
\hat f(x)=f(x)+\alpha\|x\|^d.
\]

The GEAP algorithm proposes adaptive refinement of the shift $\alpha\in\mathbb{R}$ according to the following scheme. For local maxima of~\eqref{geapopt}, let $x\in\Sigma$ and let $\tau>0$ be the threshold for positive definiteness of the Hessian matrix $\hat H(x)=\nabla^2\hat f(x)$. If we denote by $\lambda_{\rm min}(H)$ the minimal and by $\lambda_{\rm max}(H)$ the maximal eigenvalue of the Hessian matrix $H(x)=\nabla^2f(x)$, then choosing $\alpha=\max\{0,(\tau-\lambda_{\min}(H))/d\}$ will force the shifted Hessian to be positive definite with $\lambda_{\min}(\hat H)\geq\tau$. For local minima of~\eqref{geapopt}, choosing the shift $\alpha=-\max\{0,(\tau-\lambda_{\min}(-H))/d\}$ will force the shifted Hessian matrix to be negative definite with $\lambda_{\rm max}(\hat H)\leq -\tau$. With these selections, Kolda and Mayo~\cite{KoMa14} proved via fixed-point analysis that the shifted target function $\hat f$ is locally convex (respectively, concave) in the neighborhood of the trial point $x\in\mathbb{R}^n$ and, with large enough $\tau$, convergence is guaranteed.

The Hessian matrix $H(x)=\nabla^2f(x)$ has been computed analytically in~\cite{KoMa14} for general symmetric tensors $B$ and we mention the result below for completeness:
\begin{align*}
&H(x)=2\frac{d^2Ax^d}{(Bx^d)^3}(Bx^{d-1}\otimes Bx^{d-1})+\frac{d}{Bx^d}\big((d-1)Ax^{d-2}+Ax^d(I+(d-2)(x\otimes x))\\
&+d(Ax^{d-1}\otimes x+x\otimes Ax^{d-1})\big)-\frac{d}{(Bx^d)^2}\bigg((d-1)Ax^dBx^{d-2}+d\big(Ax^{d-1}\otimes Bx^{d-1}\\
&+Bx^{d-1}\otimes Ax^{d-1}\big)+dAx^d\big(x\otimes Bx^{d-1}+Bx^{d-1}\otimes x\big)\bigg),
\end{align*}
where $I$~denotes the $n\times n$ identity matrix and the outer product is defined for vectors $a,b\in\mathbb{R}^n$ by $a\otimes b=ab^\top$. The expression for the Hessian matrix can always be computed as a linear combination of the matrices $Ax^{d-2},Bx^{d-2}\in\mathbb{R}^{n\times n}$ and outer products of the vectors $Ax^{d-1},Bx^{d-1}\in\mathbb{R}^n$, which can be computed efficiently in the TT format. For instance, if $A$ is given by the TT cores $(G_k)_{k=1}^d$, then
\[
Ax^{d-2}=G_1G_2(x^\top G_3^\top)\cdots(x^\top G_{d-1}^\top)(x^\top G_d^\top)
\]
and the respective expressions for $Ax^{d-1}$ and $Ax^d$ are similar.

By the previous discussion, we can state the GEAP algorithm in terms of TT tensors.
\begin{algorithm}[TT-GEAP]\label{ttgeap}
Let $A$ and $B$ be $n$-dimensional tensors of order $d$ characterized by the TT cores $(G_k)_{k=1}^d$ and $(\Gamma_k)_{k=1}^d$, respectively. Let  $x_0\in\mathbb{R}^n$, $\|x_0\|=1$, be an initial guess. Let $\beta=1$ if we want to find local maxima; for local minima, let $\beta=-1$. Let $\tau>0$ be the tolerance. Then the GEAP algorithm for tensor-trains can be implemented in the following way.
\begin{algorithmic}
\State $x_0\leftarrow x_0/\|x_0\|$
\For{$k=0,1,\ldots$}
\State $Ax_k^d\leftarrow(x_{k}^\top G_1^\top)(x_{k}^\top G_2^\top)\cdots(x_{k}^\top G_{d-1}^\top)(x_{k}^\top G_d^\top)$
\State $Ax_k^{d-1}\leftarrow G_1(x_{k}^\top G_2^\top)\cdots(x_{k}^\top G_{d-1}^\top)(x_{k}^\top G_d^\top)$
\State $Ax_k^{d-2}\leftarrow G_1G_2(x_{k}^\top G_3^\top)\cdots(x_{k}^\top G_{d-1}^\top)(x_{k}^\top G_d^\top)$
\State $Bx_k^d\leftarrow(x_{k}^\top \Gamma_1^\top)(x_{k}^\top \Gamma_2^\top)\cdots(x_{k}^\top \Gamma_{d-1}^\top)(x_{k}^\top \Gamma_d^\top)$
\State $Bx_k^{d-1}\leftarrow\Gamma_1(x_{k}^\top \Gamma_2^\top)\cdots(x_{k}^\top \Gamma_{d-1}^\top)(x_{k}^\top \Gamma_d^\top)$
\State $Bx_k^{d-2}\leftarrow \Gamma_1\Gamma_2(x_{k}^\top \Gamma_3^\top)\cdots(x_{k}^\top \Gamma_{d-1}^\top)(x_{k}^\top \Gamma_d^\top)$
\State $\lambda_{k}\leftarrow Ax_{k}^d/Bx_{k}^d$
\State $\alpha_k\leftarrow\beta\max\{0,(\tau-\lambda_{\rm min}(\beta H(x_{k}))/d\}$
\State $\hat x_{k+1}\leftarrow\beta(Ax_{k}^{d-1}-\lambda_kBx_{k}^{d-1}+(\alpha_k+\lambda_{k})Bx_{k}^dx_{k})$
\State $x_{k+1}=\hat x_{k+1}/\|\hat x_{k+1}\|$
\EndFor
\end{algorithmic}
\end{algorithm}
\section{Lattice-theoretic tensors}\label{meetandjointensors}
Let $(P,\preceq)$ be a nonempty poset. Let $S=\{x_1,\ldots,x_n\}$ be a finite subset of $P$ such that $x_i\preceq x_j$ only if $i\leq j$ and let $f$ be a complex-valued function on $P$. The poset $P$ is said to be locally finite if the interval
\[
\{z\in P\mid x\preceq z\preceq y\}
\]
is finite for all $x,y\in P$. If the greatest lower bound of $x,y\in P$ exists, it is called the meet of $x$ and $y$ and is denoted by $x\wedge y$, and if the least upper bound of $x,y\in P$ exists, it is called the join of $x$ and $y$ and is denoted by $x\vee y$. If $x\wedge y$ exists for all $x,y\in P$, then $(P,\preceq,\wedge)$ is called a meet semilattice, and if $x\vee y\in P$ exists for all $x,y\in P$, then $(P,\preceq,\vee)$ is called a join semilattice. If the poset $(P,\preceq,\wedge,\vee)$ is both a meet semilattice and a join semilattice, then it is called a lattice. The poset $(\mathbb{Z}_+,|)$, where $|$ is the ordinary divisibility relation, is a locally finite lattice.

Let $(P,\preceq,\wedge)$ be a meet semilattice. Then the $n$-dimensional tensor of order $d$, $(S_d)_f$, where
\[
((S_d)_f)_{i_1,\ldots,i_d}=f(x_{i_1}\wedge\cdots\wedge x_{i_d}),\quad i_1,\ldots,i_d\in\{1,\ldots,n\},
\]
is called the order $d$ meet tensor on $S$ with respect to $f$. If $(P,\preceq,\vee)$ is a join semilattice, then the $n$-dimensional tensor of order $d$, $[S_d]_f$, where
\[
([S_d]_f)_{i_1,\ldots,i_d}=f(x_{i_1}\vee\cdots\vee x_{i_d}),\quad i_1,\ldots,i_d\in\{1,\ldots,n\},
\]
is called the order $d$ join tensor on $S$ with respect to $f$.

It is known that one can induce an arbitrary commutative, associative, and idempotent meet operator $\wedge$ (respectively, join operator $\vee$) by defining an appropriate partial ordering $\preceq$~on $P$. Since the partial ordering $\preceq$ and function $f\!:P\to\mathbb{C}$~can be chosen freely, the set of lattice-theoretic tensors is easily seen to contain every symmetric $n$-dimensional order $d$ tensor $A$ that satisfies the criterion
\[
A_{i_1,i_2,\ldots,i_d}=A_{j_1,j_2,\ldots,j_d}\quad\text{if }\{i_1,i_2,\ldots,i_d\}=\{j_1,j_2,\ldots,j_d\}
\]
for all $i_1,i_2,\ldots,i_d,j_1,j_2,\ldots,j_d\in\{1,2,\ldots,n\}$.

In our numerical experiments, we consider GCD and LCM tensors as special cases of meet and join tensors. In the following, let 
$$(x,y)=\max\{z\in\mathbb{Z}_+: z|x\text{ and }z|y\}
$$ 
denote the greatest common divisor and 
$$
[x,y]=\min\{z\in\mathbb{Z}_+: x|z\text{ and }y|z\}
$$
the least common multiple of $x,y\in\mathbb{Z_+}$, respectively. Let $S=\{x_1,\ldots,x_n\}$ be a finite subset of $\mathbb{Z}_+$ and let $f$ be a complex-valued function on $\mathbb{Z}_+$. The $n$-dimensional tensor of order $d$, $(S_d)_f$, where
\[
((S_d)_f)_{i_1,\ldots,i_d}=f((x_{i_1},\ldots,x_{i_d})),\quad i_1,\ldots,i_d\in\{1,\ldots,n\},
\]
is called the order $d$ GCD tensor on $S$ with respect to $f$, and the $n$-dimensional tensor of order $d$, $[S_d]_f$, where
\[
([S_d]_f)_{i_1,\ldots,i_d}=f([x_{i_1},\ldots,x_{i_d}]),\quad i_1,\ldots,i_d\in\{1,\ldots,n\},
\]
is called the order $d$ LCM tensor on $S$ with respect to $f$. The subscript $f$ is omitted if the incidence function is the identity $I(x)=x$ on $\mathbb{Z_+}$. 

The case $d=2$ corresponding to GCD and LCM matrices has been studied extensively by many authors, starting in 1876, when Smith~\cite{Smith} calculated the determinant of the so-called Smith matrix, i.e., of the $n\times n$ matrix having the greatest common divisor of $i$ and $j$ as its $ij$ entry. After that, several related papers have been published in the literature, see~\cite{Ha2,Is2,Sa}. Beslin and Ligh~\cite{BeLi} showed that GCD matrices are positive definite, and Ovall~\cite{Ov} considered positive definiteness of GCD and related matrices. Balatoni~\cite{Fe} estimated the smallest and the largest eigenvalue of the Smith matrix.  Hong and Loewy~\cite{AR} examined the asymptotic behavior of eigenvalues of power GCD matrices. Ilmonen, Haukkanen, and Merikoski~\cite{i1} examined the eigenvalues of meet and join matrices with respect to $f.$  Recently Mattila and Haukkanen~\cite{Mattila4} studied positive definiteness and eigenvalues of meet and join matrices, and Mattila, Haukkanen, and M\"{a}ntysalo~\cite{Mattila5} considered singularity of LCM-type matrices.

The study of tensors actually started even earlier than that of GCD matrices, when Cayley~\cite{Cayley} considered hyperdeterminants in 1845. The focus of the lattice-theoretic community is currently shifting to tensors as well: Haukkanen~\cite{Ha3} considered the (Cayley) hyperdeterminants of GCD tensors and Luque~\cite{Luque} studied (Cayley) hyperdeterminants of meet tensors. Ilmonen~\cite{ilhyp} considered eigenvalues of meet tensors. In the following, we give the fundamental tensor eigenvalue bound of~\cite{Qi} adapted to meet tensors.

\begin{theorem}[cf.~\cite{Qi}]\label{meetupperbound}
Let $(P,\preceq,\wedge)$ be a meet semilattice and $S=\{x_1,\ldots,x_n\}$ with $x_i\preceq x_j$ only if $i\leq j$, be a finite subset of $P$. Let $f$ be any real-valued function on $P$. Then every eigenvalue of the $n$-dimensional order $d$~meet tensor $(S_d)_f$ lies in the region
\[
\bigcup_{k=1}^n\{z\in\mathbb{C}: |z-f(x_k)|\leq \sum_{\substack{i_2,\ldots,i_d=1\\ \delta_{k,i_2,\ldots,i_d}=0}}^n|f(x_{i_k}\wedge x_{i_2}\wedge\cdots\wedge x_{i_d})|\},
\]
where $\delta$ represents the Kronecker tensor.
\end{theorem}

In this paper, we consider GCD tensors as a special case of meet tensors since the largest inclusion region of Theorem~\ref{meetupperbound} turns out to be extremely sharp for these tensors when $d\gg 2$. We also note that Theorem~\ref{meetupperbound} can be used to bound Z-eigenvalues.

We illustrate Theorem~\ref{meetupperbound} by considering as an example the full spectrum of eigenvalues of the so-called Smith tensor $A$ of dimension $n=4$ and orders $d=3,4$ defined by
\[
A_{i_1,\ldots,i_d}=(i_1,\ldots,i_d),\quad i_1,\ldots,i_d\in\{1,\ldots,n\}.
\]
Because of the small dimensionality in this example, the whole spectrum of eigenvalues can be solved by using Mathematica 10's \texttt{NSolve} function. The eigenvalues are displayed in Figures~\ref{Hsystem} and~\ref{Hsystem2}. Recall that real-valued eigenvalues are called H-eigenvalues.

\begin{figure}[!t]
\centering
\subfloat{{\includegraphics[width=6.2cm]{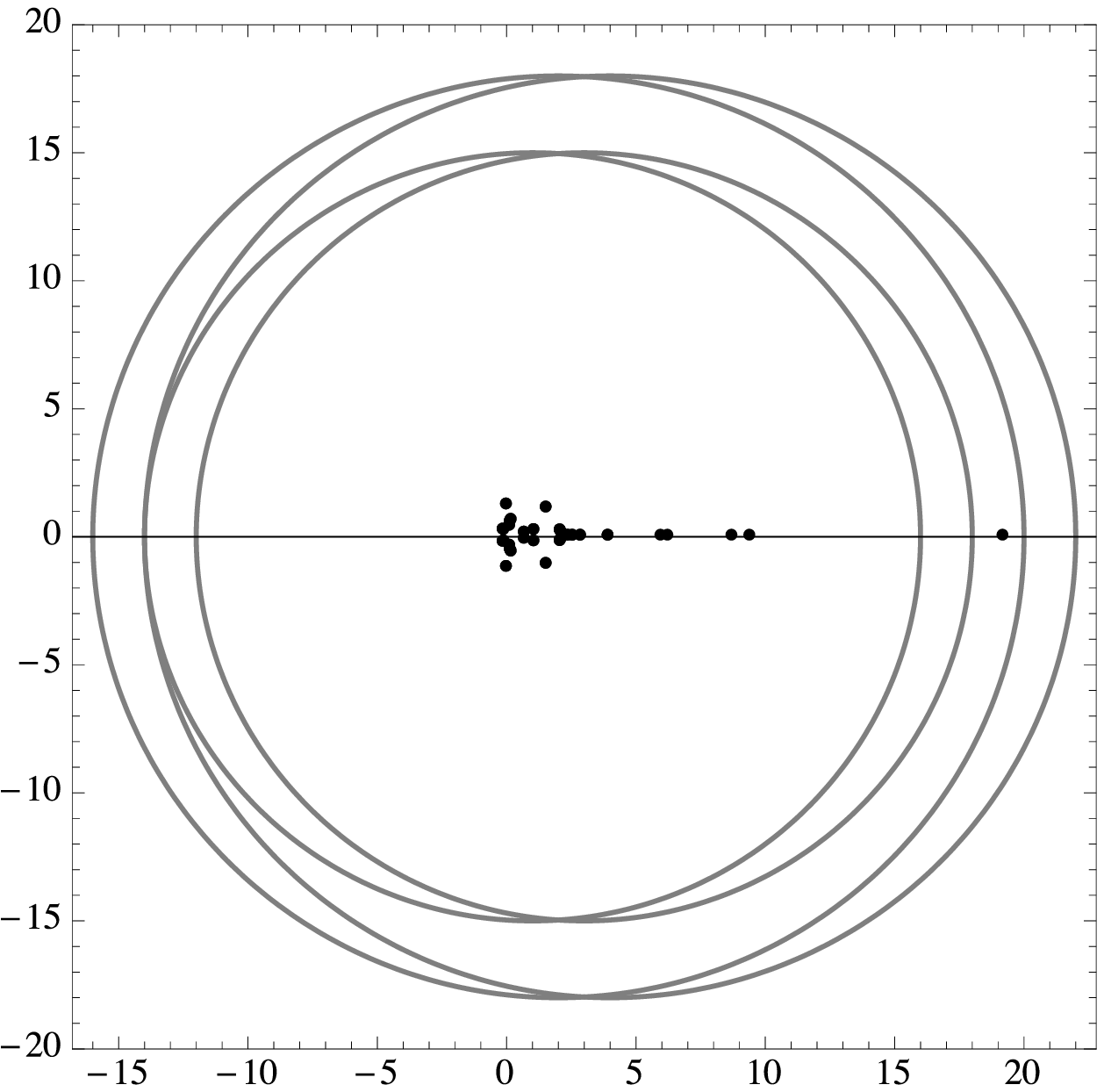}}\subfloat{{\includegraphics[width=6.2cm]{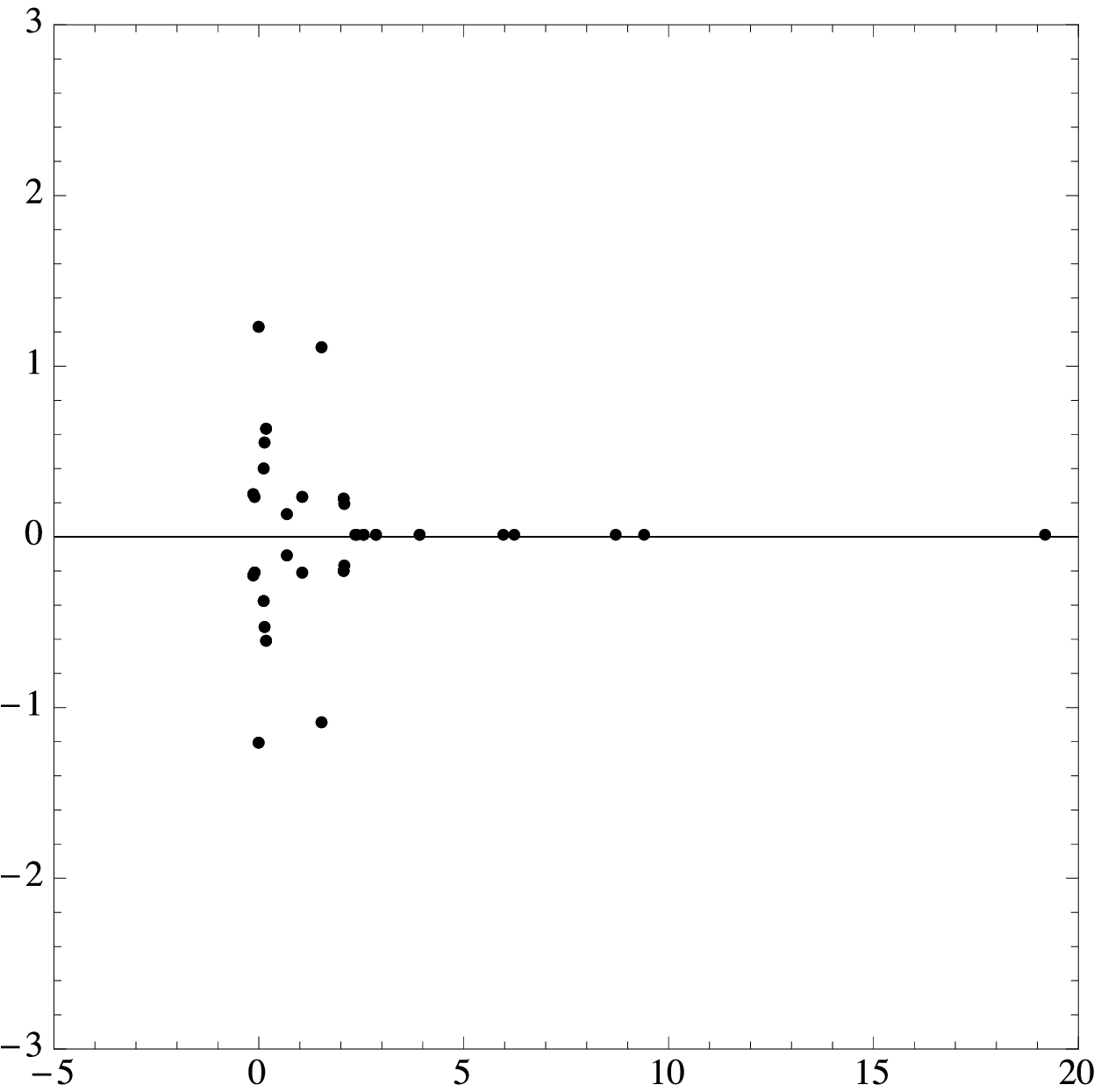}}}}
\caption{All eigenvalues of the Smith tensor of dimension 4 and order 3 in the complex plane. The contours denote the constituent disks of Theorem~\ref{meetupperbound}.}\label{Hsystem}
\subfloat{{\includegraphics[width=6.2cm]{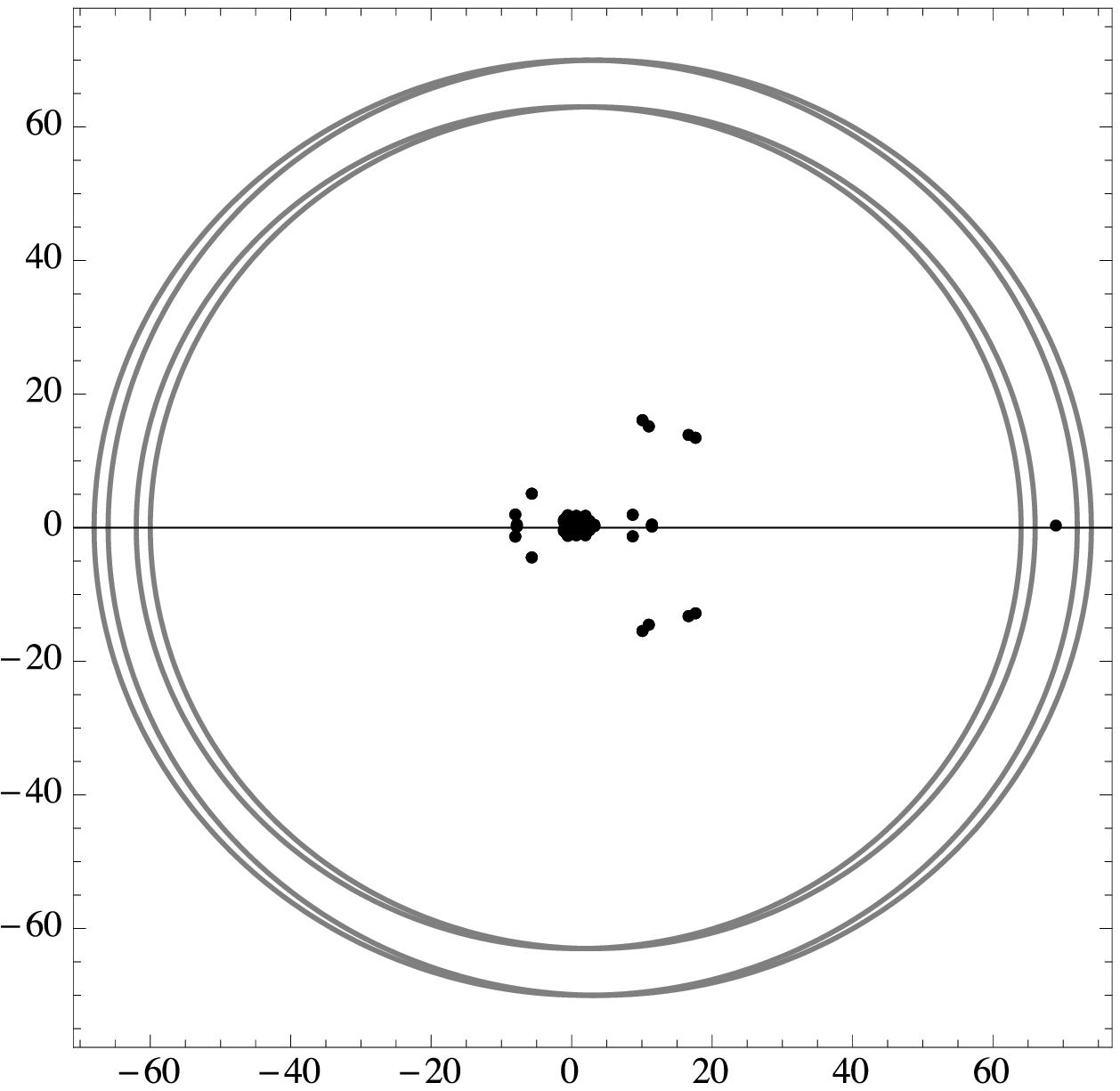}}\subfloat{{\includegraphics[width=6.2cm]{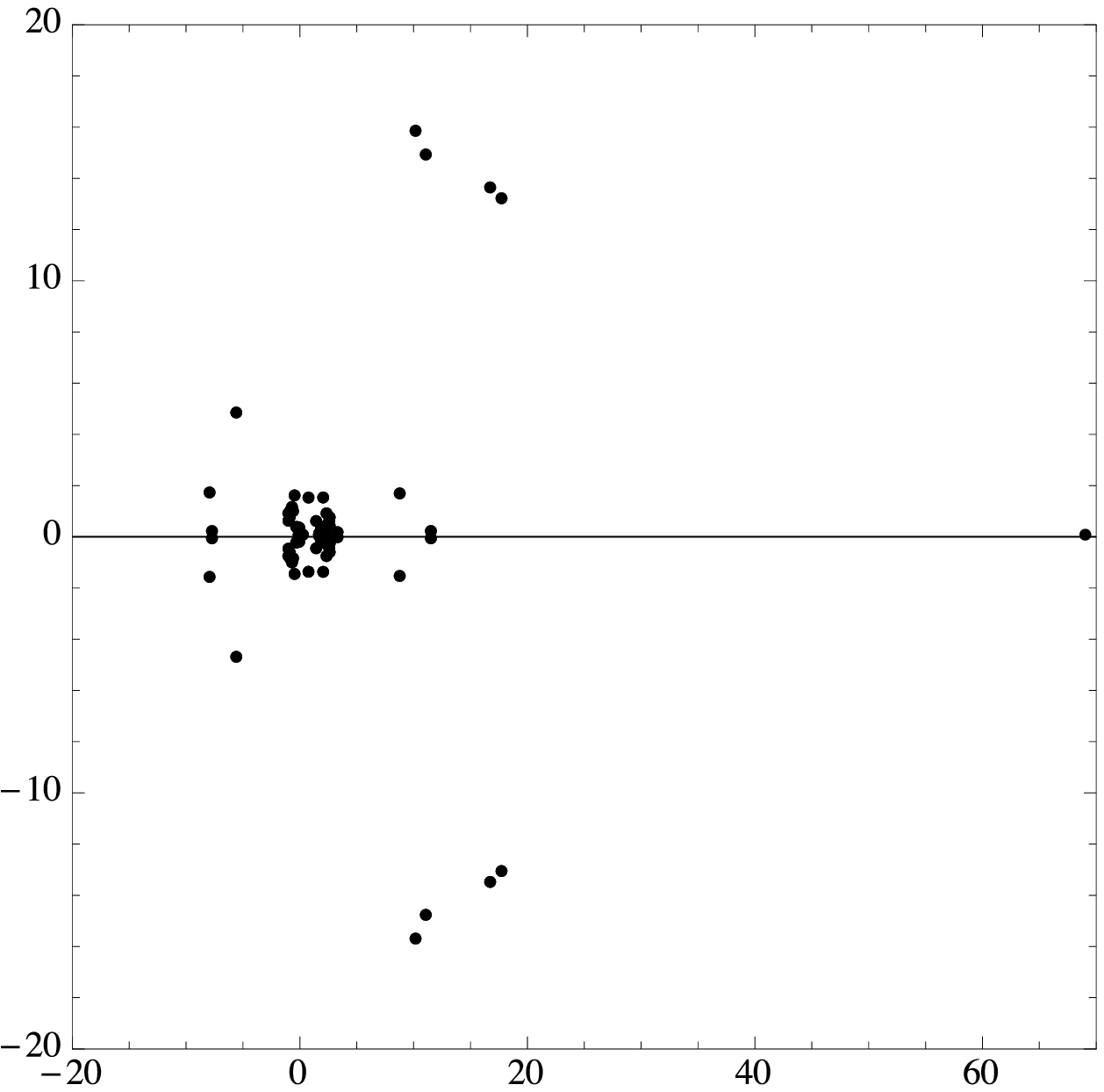}}}}
\caption{All eigenvalues of the Smith tensor of dimension 4 and order 4 in the complex plane. The contours denote the constituent disks of Theorem~\ref{meetupperbound}.}\label{Hsystem2}
\end{figure}

The parity of the order $d$ affects the distribution of non-dominant H-eigenvalues: for even $d$, the non-dominant H-eigenvalues cluster near the origin whereas for odd $d$, the non-dominant H-eigenvalues tend to cluster away from the origin. This phenomenon appears to become more prominent with increasing dimensionality: in the numerical experiments of Section~\ref{numex}, we observe that the minimal H-eigenvalues of even-ordered Smith tensors tend to zero at an exponential rate. Additional numerical studies conducted by the authors suggest that the H-eigenvalues of odd-ordered Smith tensors cluster strongly near integer values for high-dimensional problems, even becoming indistinguishable from the respective integer up to numerical precision. We hope to return to this problem once the theoretical properties of odd-ordered tensors are more developed.

\subsection{Explicit TT decomposition of meet tensors}
\label{meetTT}
Let $(P,\preceq,\wedge,\hat 0)$ be a locally finite meet semilattice with the smallest element $\hat 0\in P$ such that $\hat 0\preceq x$ for all $x\in P$. Suppose that $S=\{x_1,\ldots,x_n\}$ is a finite, meet closed subset of $P$ ordered so that $x_i\preceq x_j$ only if $i\leq j$, and that $f$ is a complex-valued function on $P$. It has been shown in~\cite{ilhyp} that the meet tensor $(S_d)_f$ has an explicit canonical decomposition given by
\begin{equation}
((S_d)_f)_{i_1,\ldots,i_d}=\sum_{k=1}^nD_k E_{i_1,k}\cdots E_{i_d,k},\quad i_1,\ldots,i_d\in\{1,\ldots,n\},\label{candecomp}
\end{equation}
where the coefficients are given by
\[
D_k=\sum_{\substack{z\preceq x_k\\Êz\not\preceq x_j,\ j<k}}\sum_{\hat 0\preceq y\preceq z}f(y)\mu_P(y,z),\quad k\in\{1,\ldots,n\},
\]
the $n\times n$ matrix $E$ is defined elementwise by
\[
E_{i,j}=\begin{cases}1&\text{if }x_j\preceq x_i,\\ 0&\text{otherwise,}\end{cases}\quad i,j\in\{1,\ldots,n\},
\]
and $\mu_P$ is the M\"{o}bius function of $P$, which can be computed inductively by
\begin{align*}
\mu_P(x,y)=\begin{cases}1&\text{if }x=y,\\ -\sum_{x\preceq z\prec y}\mu_P(x,z)&\text{if }x\prec y,\\ 0&\text{otherwise.}\end{cases}
\end{align*}
The canonical decomposition can be embedded into tensor-train formalism as follows.
\begin{proposition}\label{meettheorem}
Let $(P,\preceq,\wedge,\hat 0)$ be a meet semilattice, $S=\{x_1,\ldots,x_n\}$ a finite, meet closed subset of $P$ ordered $x_i\preceq x_j$ only if $i\leq j$, and $f$ a complex-valued function on $P$. Then the $n$-dimensional meet tensor $(S_d)_f$ of order $d$ has a rank-$n$ tensor-train decomposition
\[
((S_d)_f)_{i_1,\ldots,i_d}=G_1(i_1)\left(\prod_{k=2}^{d-1}G(i_k)\right)G_d(i_d),
\]
where the TT cores $G_1$, $G$, and $G_d$ are defined by setting
\begin{align}
G_1(i)_{1,j}=D_jE_{i,j},\quad G(i)_{j,k}=\delta_{j,k}E_{i,k},\quad\text{and}\quad G_d(i)_{j,1}=E_{i,j}.\label{meetttdecomp}
\end{align}
for all $i,j,k\in\{1,\ldots,n\}$. The number of nonzero elements in each core is equal to
\[
\sum_{z\in S}\#\{y\in S: y\preceq z\}.
\]
\end{proposition}

\emph{Remark.}  We use the TT decomposition --- instead of the canonical decomposition --- in this paper since it conforms to the more general framework covered by the eigenvalue algorithms discussed in Section~\ref{TTalgorithms}.

In the special case of GCD tensors in the lattice $(\mathbb{Z_+},|)$, the M\"{o}bius function $\mu_P$ is characterized by the arithmetic M\"{o}bius function $\mu_P(x,y)=\mu(y/x)$ for all $x|y$, where
\[
\mu(x)=\begin{cases}1&\text{if }x=1,\\ (-1)^n&\text{if $x$ is the product of $n$ distinct prime numbers},\\ 0&\text{otherwise.}\end{cases}
\]
For the Smith tensor, we obtain the following as a corollary.
\begin{corollary}
Let $S=\{1,\ldots,n\}$ be a subset of $\mathbb{Z}_+$. Then the number of nonzero elements in each one of the TT cores $G_1$, $G$, and $G_d$ corresponding to $(S_d)$ increases as
\[
{\rm nz}(n)=n\ln n+(2\gamma-1)n+\mathcal{O}(\sqrt{n})\quad\text{as }n\to\infty,
\]
where $\gamma$ denotes the Euler--Mascheroni constant.
\end{corollary}
\begin{proof}
The claim follows from the average order of the number-of-divisors function $\sigma_0(n)=\sum_{k|n}1$, which is known to be (see~\cite[Theorem~320]{HW}) 
\[
\frac{1}{n}\sum_{k=1}^n\sigma_0(k)=\ln n+(2\gamma-1)+\mathcal{O}(n^{-1/2}).
\]
This yields
\[
{\rm nz}(n)=\sum_{k=1}^n\sigma_0(k)=n\ln n+(2\gamma-1)n+\mathcal{O}(\sqrt{n})\quad\text{as }n\to\infty
\]
proving the assertion.  \quad
\end{proof}

We have demonstrated that meet tensors defined on meet closed sets can be stored in tensor-train format with sparse TT cores and, in the special case of the Smith tensor, requires only the order of $n\ln n+(2\gamma-1)n+\mathcal{O}(\sqrt{n})$ parameters for the exact representation of the full tensor. Notably, the TT factors are sparse while the full Smith tensor is dense. The number of parameters of this representation is independent of the tensor order $d$, which eliminates the so-called curse of dimensionality associated with storing the full array of $n^d$ values. The actual storage requirements for storing the TT cores in Mathematica 10's \texttt{SparseArray} format are displayed in Table~\ref{meetstorage}.
\begin{table}[!h]
\begin{center}
\begin{tabular}{c|c|c}
Dimension $n$&Bit count of the triplet $(G_1,G,G_d)$&Number of nonzero elements\\
\hline $10^1$&$4.80$\,kb&$3\cdot 27$\\
$10^2$&$54.73$\,kb&$3\cdot 482$\\
$10^3$&$747.56$\,kb&$3\cdot 7\,069$\\
$10^4$&$9.52$\,Mb&$3\cdot 93\,668$\\
$10^5$&$118.01$\,Mb&$3\cdot 1\,166\,750$\\
$10^6$&$1.38$\,Gb&$3\cdot 13\,970\,034$
\end{tabular}
\caption{The storage requirements of the order independent triplets in Mathematica 10's \texttt{SparseArray} format.}\label{meetstorage}
\end{center}
\end{table}
\subsection{Implicit TT decomposition of join tensors}
\label{joinTT}
Let $(P,\preceq,\vee)$ be a join semilattice, $S=\{x_1,\ldots,x_n\}$ a finite subset of $P$ ordered $x_i\preceq x_j$ only if $i\leq j$, and let $f$ be a complex-valued function on $P$. By~\cite{Comon08}, it is known that a symmetric canonical decomposition exists for any symmetric tensor over the scalar field $\mathbb{C}$ with its canonical rank bounded from above by $\binom{d+n-1}{d}$ and, in consequence, a TT decomposition exists implicitly for the join tensor $[S_d]_f$ defined elementwise by
\[
([S_d]_f)_{i_1,\ldots,i_d}=f(x_{i_1}\vee\cdots\vee x_{i_d}),\quad i_1,\ldots,i_d\in\{1,\ldots,n\}.
\]
However, the derivation of an explicit factorization as in the case of meet tensors would require new mathematical theory beyond the scope of this paper. Since the individual elements of the tensor $[S_d]_f$ can be determined by an evaluation of the function $(x_1,\ldots,x_d)\mapsto f([x_1,\ldots,x_d])$, the TT-DMRG cross algorithm can be used instead to compute a TT decomposition numerically. Even though this approach does not preserve symmetry, the approach taken here is justified by its superior computational simplicity.

Whenever one is using the TT-DMRG cross algorithm, there is a chance that the curse of dimensionality is invoked if the compression ranks associated with the TT decomposition of join tensors are prohibitively high. To this end, we prove in a special case of LCM tensors that a representative TT decomposition exists with a low a priori known maximal TT rank, thus bounding the total number of parameters needed for representation in the TT format.
\begin{theorem}\label{lcmranks}
The maximal TT rank of the $n$-dimensional order $d$ LCM tensor $A=[S_d]$ with $S=\{1,\ldots,n\}$ is given by
\[
{\rm rank}_{\rm TT}A=\begin{cases}
n&\text{if }d=3,\\
\#\{i_1\cdots i_{\lfloor d/2\rfloor }\mid 1\leq i_j\leq n\text{ and }(i_j,i_k)=1\text{ for all }j\neq k\}&\text{if }d>3.\end{cases}
\]
\end{theorem}

Before we prove Theorem~\ref{lcmranks}, we establish an auxiliary result linking certain LCM sets to multiplication tables of pairwise coprime numbers.
\begin{lemma}\label{lcmsetlemma}
Let $k\geq 2$. Then
\[
\{[i_1,\ldots,i_k]\mid 1\leq i_j\leq n\}=\{i_1\cdots i_k\mid 1\leq i_j\leq n\text{ and }(i_j,i_s)=1\text{ for all }j\neq s\}.
\]
\end{lemma}
\begin{proof}
We begin by showing that the latter set is contained in the former. In the case $k=2$, it follows from the well-known identity $[i_1,i_2]=i_1i_2/(i_1,i_2)$ that $[i_1,i_2]=i_1i_2$, if and only, if $(i_1,i_2)=1$. It is a simple exercise to show by induction with respect to $k$ that $[i_1,\ldots,i_k]=i_1\cdots i_k$, if and only, if $(i_j,i_s)=1$~for all $j\neq s$. This resolves the inclusion of the latter set into the former.

To show the converse, we consider the prime number decomposition of the LCM. Let
\[
i_j=\prod_{p\text{ prime}}p^{i_j^{(p)}}
\]
for some $i_j^{(p)}\geq 0$, $j\in\{1,\ldots,n\}$. Now consider the LCM
\begin{equation}
[i_1,\ldots,i_k]=\prod_{p\text{ prime}}p^{\max\{i_j^{(p)}\mid 1\leq j\leq k\}}.\label{lcmprimedecomp}
\end{equation}
The idea of the proof is to separate the factors in~\eqref{lcmprimedecomp} according to the maximizers corresponding to each one of the arguments while simultaneously taking careful notice not to include multiple possible maximizers equal to $\max\{i_j^{(p)}\mid 1\leq j\leq k\}$ into the resulting formula. For $k\geq 2$, the equality $[i_1,\ldots,i_k]=\alpha_1\cdots\alpha_k$ holds when
\[
\alpha_j=\prod_{\substack{p\text{ prime}\\ i_j^{(p)}\geq i_s^{(p)}~\text{for all }j<s\\ \text{and }i_j^{(p)}>i_s^{(p)}~\text{for all } j>s}}p^{i_j^{(p)}},\quad 1\leq j\leq k.
\]
By construction, it holds that $1\leq\alpha_j\leq i_j$ and $(\alpha_j,\alpha_s)=1$ for all $j\neq s$.  \quad
\end{proof}

\begin{proof}(Proof of Theorem~\ref{lcmranks}) The $k^\text{th}$ unfolding matrix of $A$ has the form
\[
(A_{[k]})_{(i_1,\ldots,i_k),(i_{k+1},\ldots,i_d)}=[[i_1,\ldots,i_k],[i_{k+1},\ldots,i_d]]=[\alpha,\beta],
\]
where we have introduced $\alpha\in([i_1,\ldots,i_k])_{(i_1,\ldots,i_k)\in\{1,\ldots,n\}^k}$ to enumerate the rows and $\beta\in([i_1,\ldots,i_{d-k}])_{(i_1,\ldots,i_{d-k})\in\{1,\ldots,n\}^{d-k}}$, the columns of $A_{[k]}$. The maximal TT rank can be bounded from above by
\begin{align*}
{\rm rank}_{\rm TT}A&=\max_{1\leq k\leq d-1}{\rm rank}\,A_{[k]}\\
&\leq\max_{1\leq k\leq d-1}\min\{\#\{[i_1,\ldots,i_k]\mid 1\leq i_j\leq n\},\#\{[i_1,\ldots,i_{d-k}]\mid 1\leq i_j\leq n\}\}\\
&=\max_{1\leq k\leq\lfloor d/2\rfloor}\#\{[i_1,\ldots,i_k]\mid 1\leq i_j\leq n\}\\
&=\#\{[i_1,\ldots,i_{\lfloor d/2\rfloor}]\mid 1\leq i_j\leq n\},
\end{align*}
where the penultimate equality follows from the fact that $A_{[k]}=A_{[d-k]}^\top$ have equal ranks and the final equality stems from the fact that $[i_1,\ldots,i_k,1]=[i_1,\ldots,i_k]$, which implies that the sequence $(\#\{[i_1,\ldots,i_k]\mid 1\leq i_j\leq n\})_{k=1}^d$ is nondecreasing. 

To bound the maximal TT rank from below, we select a submatrix $B_k$ of $A_{[k]}$ by setting
\[
(B_k)_{\alpha,\beta}=[\alpha,\beta],
\]
where $\alpha,\beta\in\{[i_1,\ldots,i_{\min\{k,d-k\}}]\mid 1\leq i_j\leq n\}$. Now $B_k$ is an LCM matrix defined on a factor closed set and thus invertible~\cite[Section~3]{Smith}, i.e., it has full rank. Hence
\[
{\rm rank}_{\rm TT}A\geq {\rm rank}\,A_{[\lfloor d/2\rfloor]}\geq{\rm rank}\,B_{\lfloor d/2\rfloor}=\#\{[i_1,\ldots,i_{\lfloor d/2\rfloor}]\mid 1\leq i_j\leq n\}.
\]
The claim follows from Lemma~\ref{lcmsetlemma} by observing that the obtained maximal TT rank is equal to the cardinality of the corresponding multiplication table.

The special case $d=3$ follows by observing that the only unfolding matrix of $A=[S_3]$ up to transposition is given by the rank-$n$ matrix
\[
(A_{[1]})_{i_1,(i_2,i_3)}=[i_1,\beta],
\]
where $1\leq i_1\leq n$ and $\beta\in\{1,\ldots,n\}^2$. This concludes the proof.\quad 
\end{proof}

We have tabulated the exact ranks of the LCM tensor given by Theorem~\ref{lcmranks} in Table~\ref{exactranks}. We note that the sequence corresponding to $4\leq d\leq 5$ is precisely the sequence A027435 in the Online Encyclopedia of Integer Sequences and the sequences corresponding to $d>5$ are its natural multivariable extensions. Unfortunately, no closed form representation is known to exist for the sequences with $d>3$. In fact, even the precise asymptotical behavior of these sequences is unknown since they may be regarded as a special case of the currently open \emph{Erd\H{o}s multiplication table problem}.

\begin{table}[!h]
\begin{center}
\begin{tabular}{c|cccccc}
&$n=2$&$n=3$&$n=4$&$n=5$&$n=6$&$n=7$\\
\hline $d=3$&$2$&$3$&$4$&$5$&$6$&$7$\\
$d=4$&$2$&$4$&$6$&$10$&$11$&$17$\\
$d=5$&$2$&$4$&$6$&$10$&$11$&$17$\\
$d=6$&$2$&$4$&$6$&$12$&$12$&$23$\\
$d=7$&$2$&$4$&$6$&$12$&$12$&$23$\\
$d=8$&$2$&$4$&$6$&$12$&$12$&$24$
\end{tabular}
\end{center}
\caption{Maximal TT ranks of the $n$-dimensional order $d$~LCM tensor $A=[S_d]$ by Theorem~\ref{lcmranks}.}\label{exactranks}
\end{table}

\section{Numerical experiments}\label{numex}
The Algorithms~\ref{ttshopm} and~\ref{ttgeap} were utilized in the solution of several tensor eigenvalue problems concerning GCD and LCM tensors. The GCD (Smith) tensors $A=(S_d)$ with $S=\{1,\ldots,n\}$ were generated using the representation of Proposition~\ref{meettheorem}. In this special case, the coefficients are given by Euler's totient function $D_k=\phi(k)$ for $k\in\{1,\ldots,n\}$ and the partial ordering is given by the ordinary divisibility relation $\preceq=|$. The LCM tensors were constructed using the TT-DMRG cross algorithm with local errors not exceeding $10^{-14}$~in the Frobenius norm.

In the case of GCD tensors, Proposition~\ref{meettheorem} permits writing the contractions using the compact formulae
\[
Ax^{d-1}=G_1(x^\top G^\top)^{d-2}(x^\top G_d^\top)\quad\text{and}\quad Ax^d=(x^\top G_1^\top)(x^\top G^\top)^{d-2}(x^\top G_d^\top),
\]
where the vector-tensor product is sparse yielding substantial computational savings.

We note that the use of TT-S-HOPM for the solution of the largest Z-eigenvalue in Subsection~\ref{dominantex} is justified since $f(x)=Ax^d$ is convex for all $x\in\mathbb{R}^n$. This can be easily verified by considering the decomposition~\eqref{candecomp}, which gives an expression for the Hessian matrix
\[
(H(x))_{i,j}=(\nabla^2f(x))_{i,j}=d(d-1)\sum_{k=1}^n\phi(k)(x^\top E_{\cdot,k})^{d-2}E_{i,k}E_{j,k}
\]
for all $i,j\in\{1,\ldots,n\}$ and $x\in\mathbb{R}^n$. Then for any $y\in\mathbb{R}^n$ we obtain
\begin{equation}
y^\top H(x)y=d(d-1)\sum_{k=1}^n\phi(k)(x^\top E_{\cdot,k})^{d-2}(y^\top E_{\cdot,k})^2\geq 0\label{convexity}
\end{equation}
whenever the order $d$ is even. Since the Hessian matrix is positive semidefinite for all $x\in\mathbb{R}^n$, it follows that $f$ is convex and TT-S-HOPM terminates. By setting $y=x$ in~\eqref{convexity}, it can also be verified that the Smith tensor $A$ is positive definite since $p(x)=Ax^d>0$ for all $x\in\mathbb{R}^n\setminus\{0\}$. Conditions for the positive definiteness of general meet tensors have been discussed in~\cite{ilhyp}.

\subsection{Dominant eigenvalues using TT-S-HOPM}\label{dominantex}
The dominant H- and Z-eigenvalues of the Smith tensor $A=(S_d)$, where $S=\{1,\ldots,n\}$, were computed using the TT-S-HOPM algorithm. Let $A$ be defined by the TT factors $G_1$, $G$, and $G_d$ as in Proposition~\ref{meettheorem}. The positive definiteness of the Smith tensor together with the upper bound imposed by Theorem~\ref{meetupperbound} imply for the H- and Z-eigenvalues of Smith tensors that
\[
0<\lambda\leq {\rm max}(G_1(\textbf{1}^\top G^\top)^{d-2}(\textbf{1}^\top G_d^\top)),\quad \textbf{1}=(1,\ldots,1)^\top\in\mathbb{R}^n,
\]
where the maximum is taken over the individual elements of the vector argument.

An ensemble of Smith tensors was generated using the sparse TT decomposition \eqref{meetttdecomp} and the H- and Z-eigenvalues were computed using
\begin{itemize}
\item[(i)] Moderate even order between $2\leq d\leq 20$ (Figure~\ref{dominantHeigenvalues}).
\item[(ii)] High even order with logarithmically spaced values between $22\leq d\leq 1000$ (Figure~\ref{dominantHeigenvalues2}).
\end{itemize}
To ensure that even with the highest dimensional tensors the accuracy was obtained, the computations were carried out by using long floating point numbers (1\,000 digit precision).

As initial guesses, we generated uniformly random $x_0\in [0,1]^n$ for H-eigenvalues and $x_0\in [-1,1]^n$ for Z-eigenvalues. TT-S-HOPM was terminated once either $20$ iterations had been completed, in which case the eigenvalue was rejected, or $|\lambda_{k}-\lambda_{k-1}|<10^{-14}$, in which case the eigenvalue was kept. For Z-eigenvalues, the test was repeated using $50$ different initial guesses for each pair $(d,n)$ to ensure that TT-S-HOPM did not converge to other local maxima. In every trial, TT-S-HOPM converged to the same Z-eigenvalue.

The dominant H-eigenvalues increase at the same rate as the theoretical upper bound imposed by Theorem~\ref{meetupperbound}, and the ratio between the upper bound and obtained dominant H-eigenvalues tends to unity with increasing $d$~and $n$. Increasing the tensor order thus improves the theoretical upper bound relative to the magnitude of the dominant H-eigenvalue. The largest Z-eigenvalues increase at a slower rate than H-eigenvalues and the theoretical estimate for H-eigenvalues has been tabulated alongside for comparison. The ratio of the upper bound of Theorem~\ref{meetupperbound} with respect to the largest Z-eigenvalues is consistently of the order of $n^{d/2-1}$.
\begin{figure}[!h]
\centering
\subfloat{{\includegraphics[height=5.5cm]{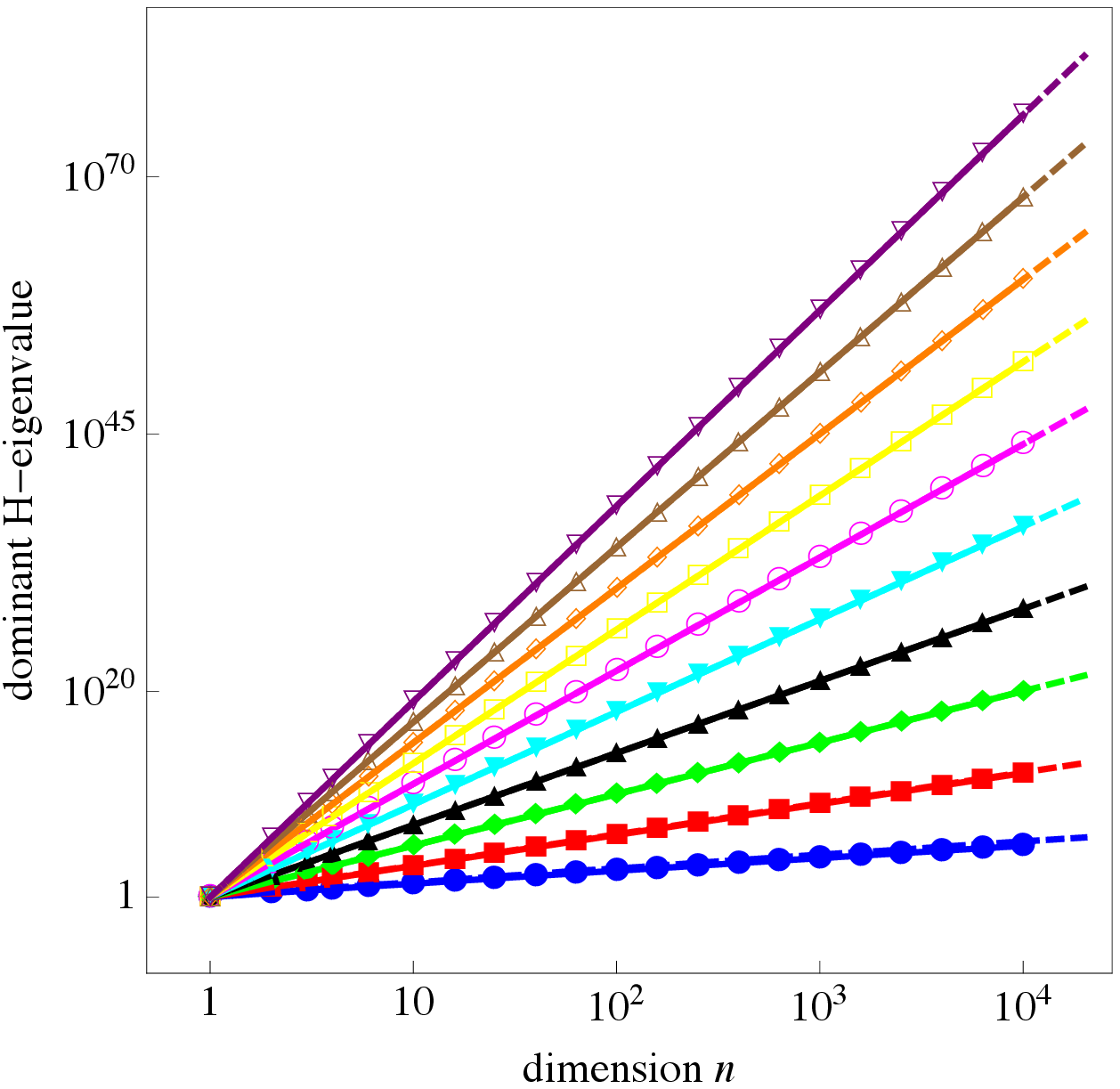}}}
\subfloat{{\includegraphics[height=5.5cm]{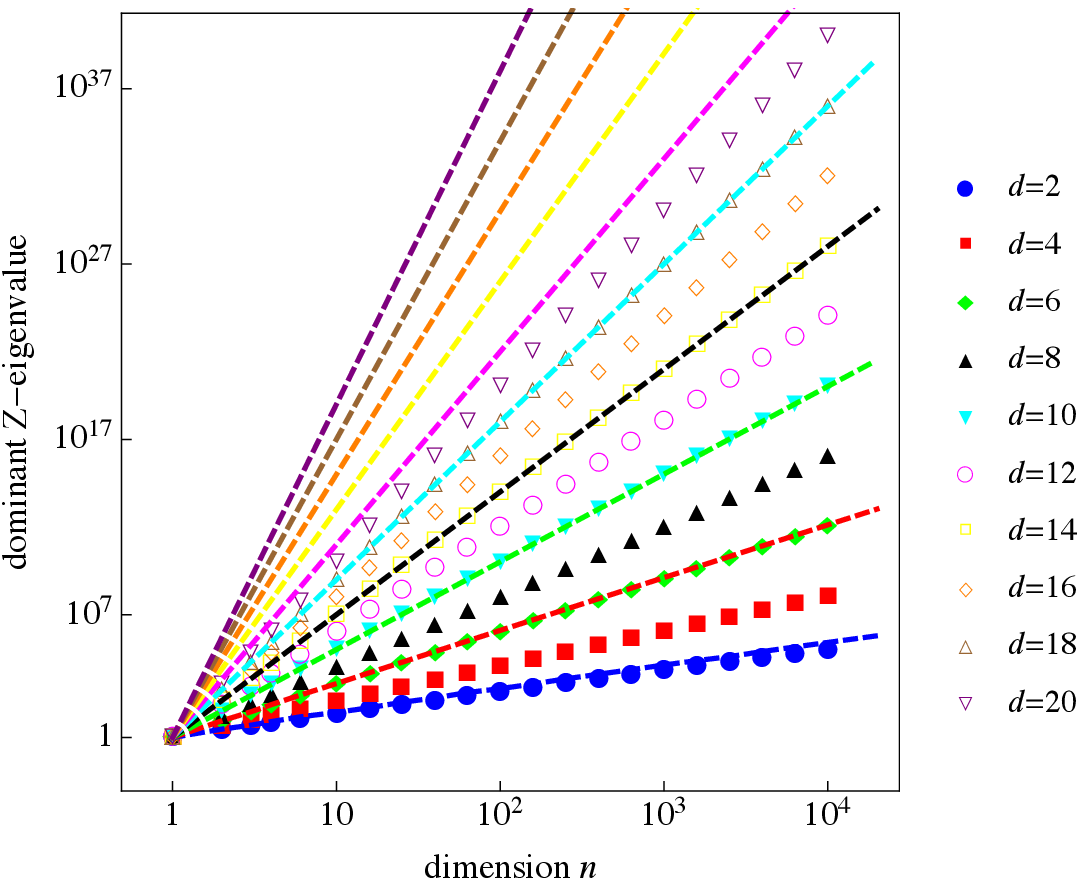}}}
\caption{Left: Dominant H-eigenvalues of the Smith tensor as functions of the dimension $n$. Right: Dominant Z-eigenvalues of the Smith tensor as functions of the dimension $n$. The dashed lines correspond to the theoretical bound imposed by Theorem~\ref{meetupperbound} on the H-eigenvalues corresponding to the same color. The bound for H-eigenvalues is also tabulated alongside Z-eigenvalues to highlight the significantly different rate of growth.}\label{dominantHeigenvalues}
\end{figure} 
\begin{figure}[!h]
\centering
\subfloat{{\includegraphics[height=5.4cm]{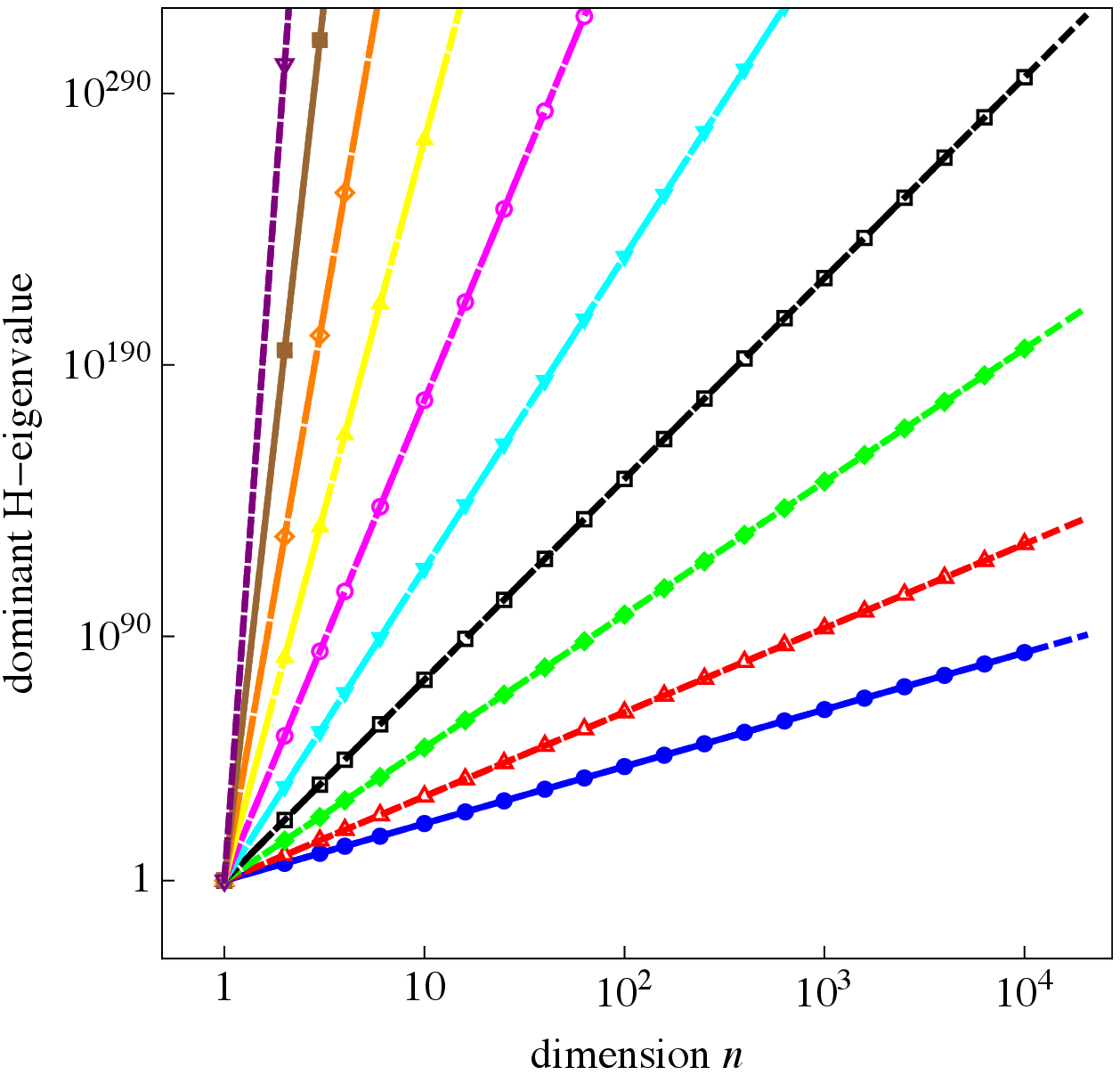}}}
\subfloat{{\includegraphics[height=5.4cm]{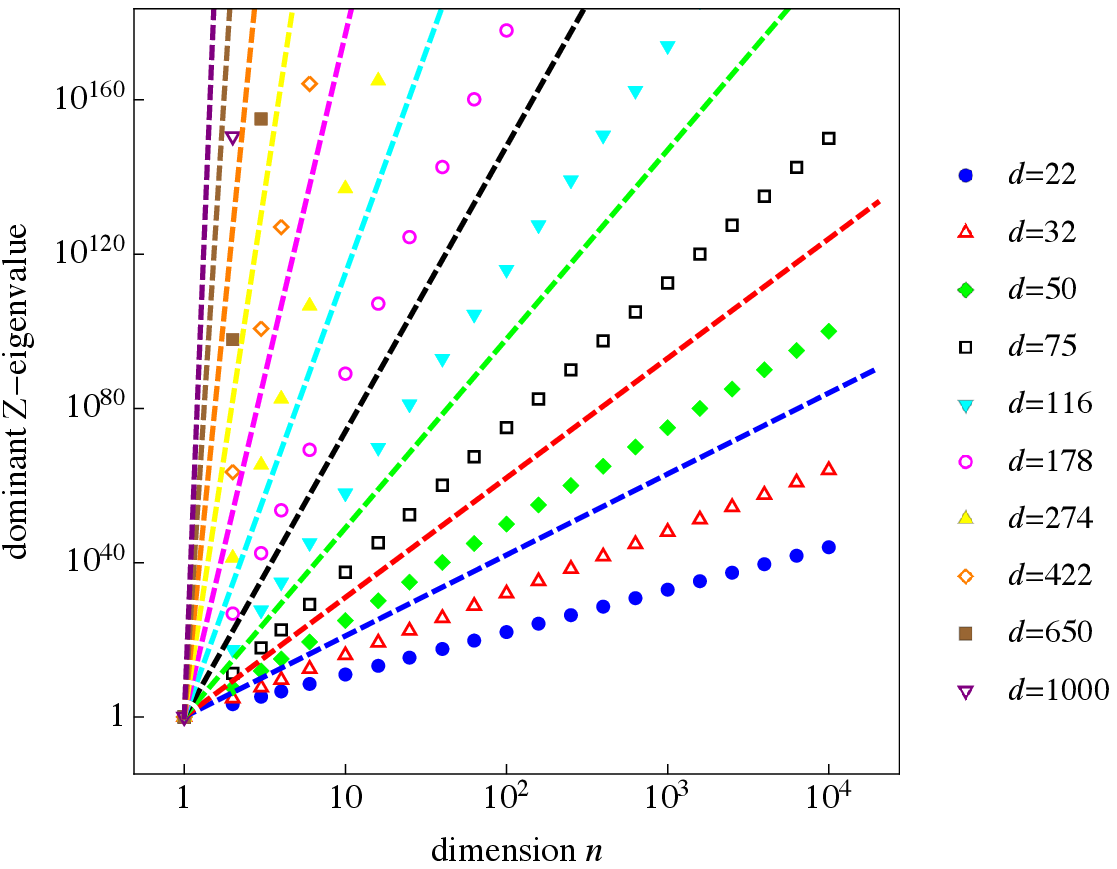}}}
\caption{Left: Dominant H-eigenvalues of the Smith tensor as functions of the dimension $n$. Right: Dominant Z-eigenvalues of the Smith tensor  as functions of the dimension $n$. The dashed lines correspond to the theoretical bound imposed by Theorem~\ref{meetupperbound} on the H-eigenvalues corresponding to the same color. The bound on H-eigenvalues is also tabulated alongside Z-eigenvalues to highlight the significantly different rate of growth.}\label{dominantZeigenvalues}\label{dominantHeigenvalues2}
\end{figure} 
\subsection{Minimal eigenvalues using TT-GEAP}\label{minimaleigs}
The TT-GEAP algorithm enables the computation of minimal H- and Z-eigenvalues. We continue with the study of the Smith tensor $A=(S_d)$, where $S=\{1,\ldots,n\}$. Unlike the dominant eigenvalues, not much is known about the behavior of minimal eigenvalues a priori other than they are positive. The results are displayed in two figures:
\begin{itemize}
\item[(iii)] The eigenvalues are displayed for fixed dimension $n$ and increasing order $d$ (Figure~\ref{minimalHeigenvalues1}).
\item[(iv)] The eigenvalues are displayed for fixed order $d$ and increasing dimension $n$ (Figure~\ref{minimalHeigenvalues2}).
\end{itemize}
\begin{figure}[!t]
\centering
\subfloat{{\includegraphics[height=.429\textwidth]{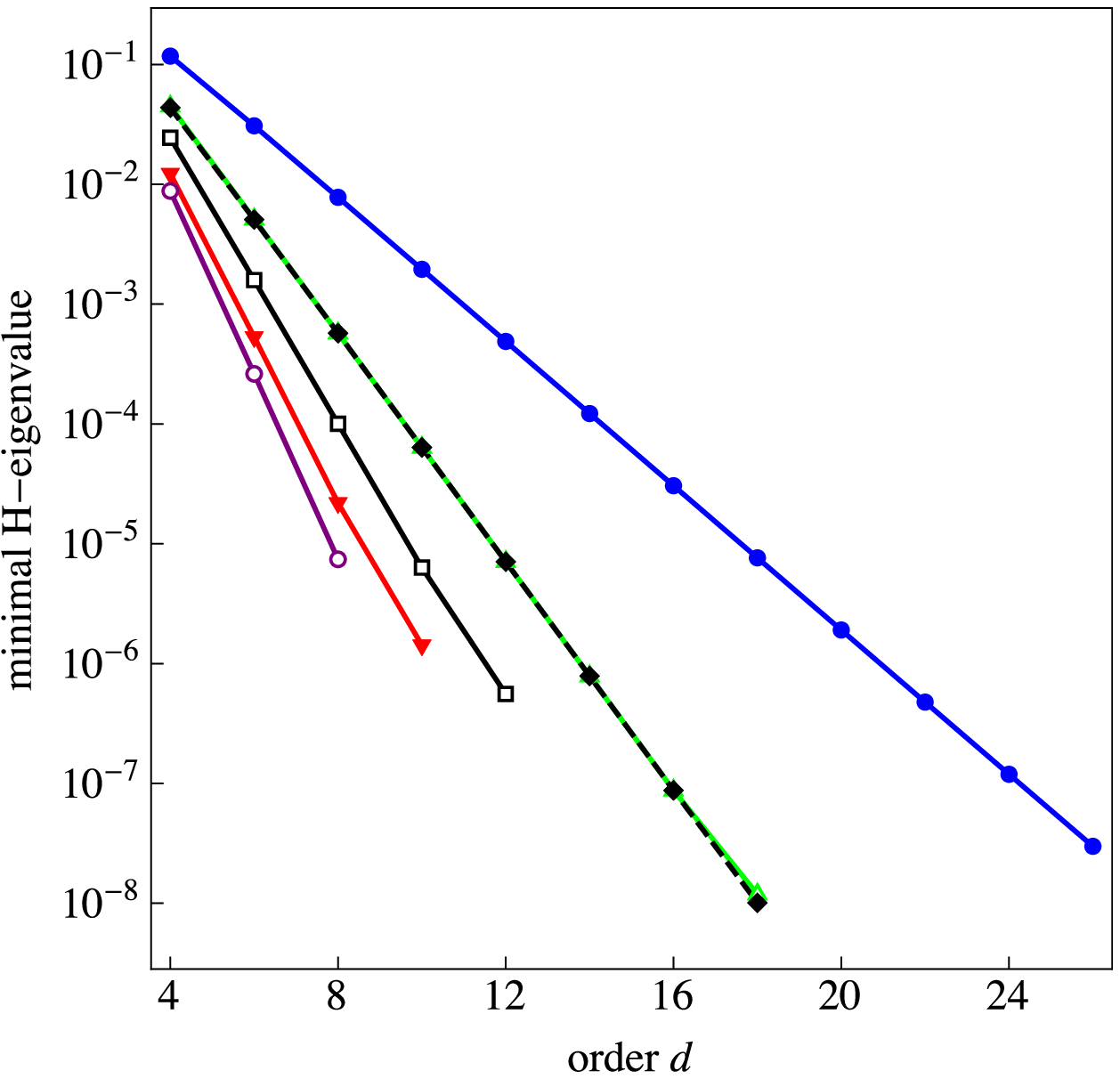}}}\subfloat{{\includegraphics[height=.429\textwidth]{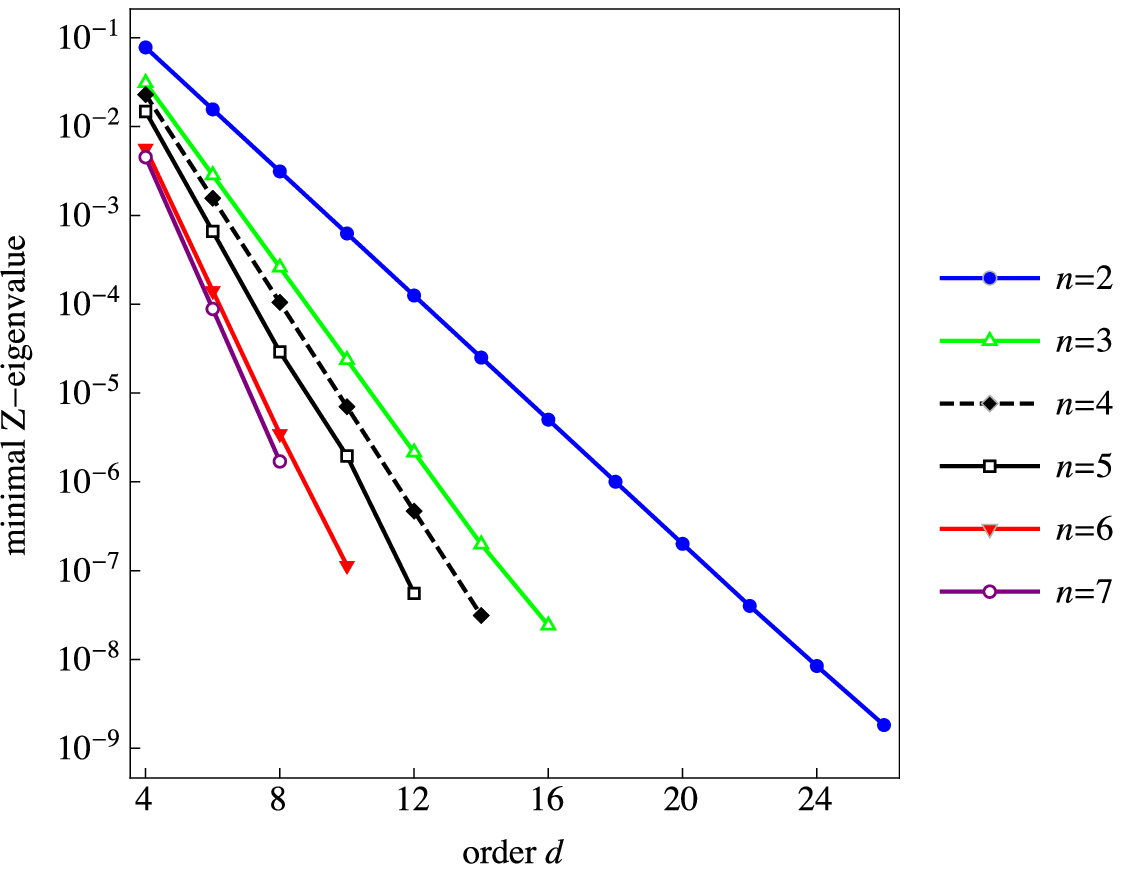}}}
\caption{Left: Minimal H-eigenvalues of the Smith tensor as functions of the order $d$. Right: Minimal Z-eigenvalues of the Smith tensor as functions of the order $d$.}\label{minimalHeigenvalues1}
\end{figure}
\begin{figure}[!t]
\centering
\subfloat{{\includegraphics[height=.44\textwidth]{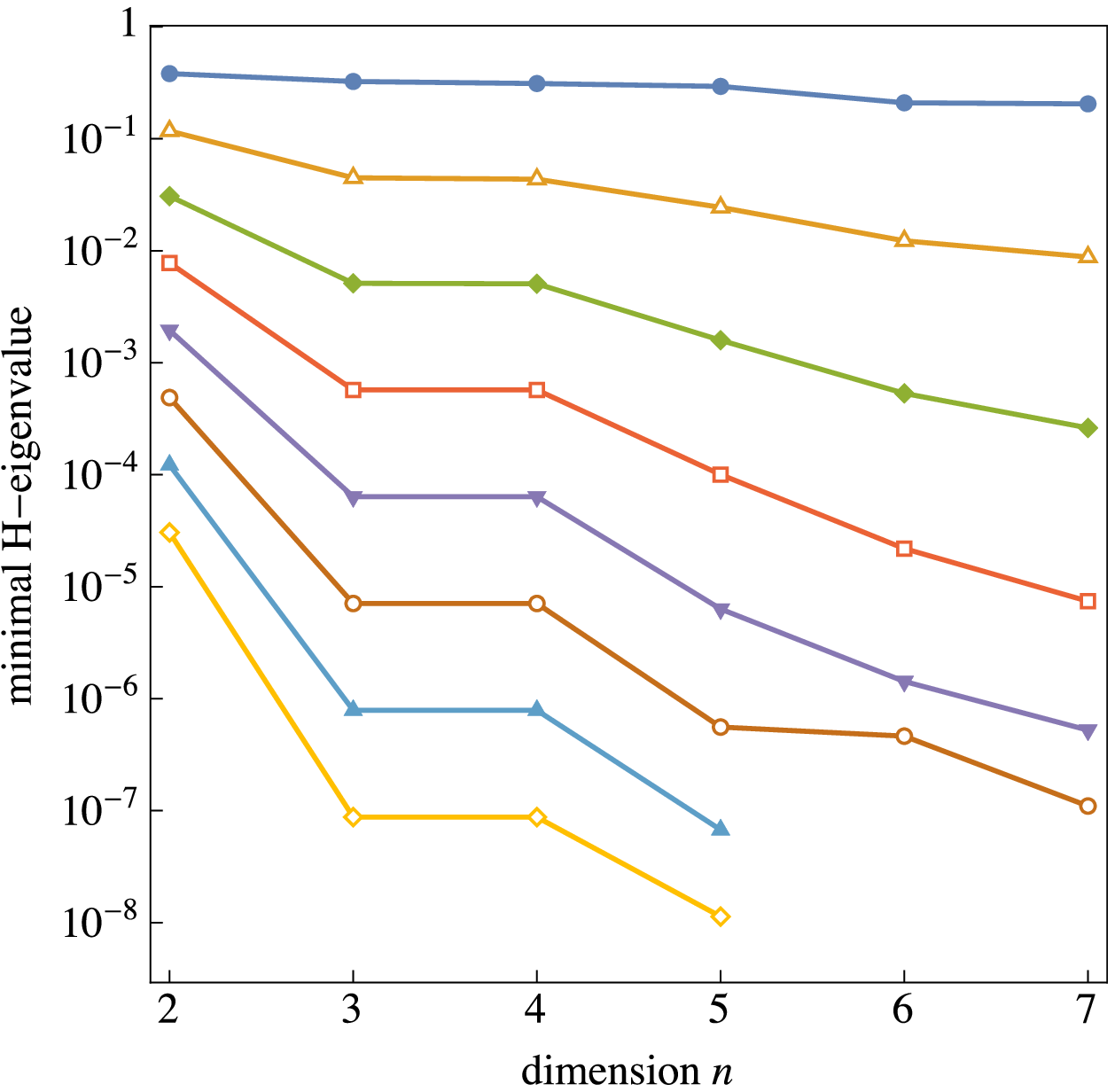}}}\subfloat{{\includegraphics[height=.442\textwidth]{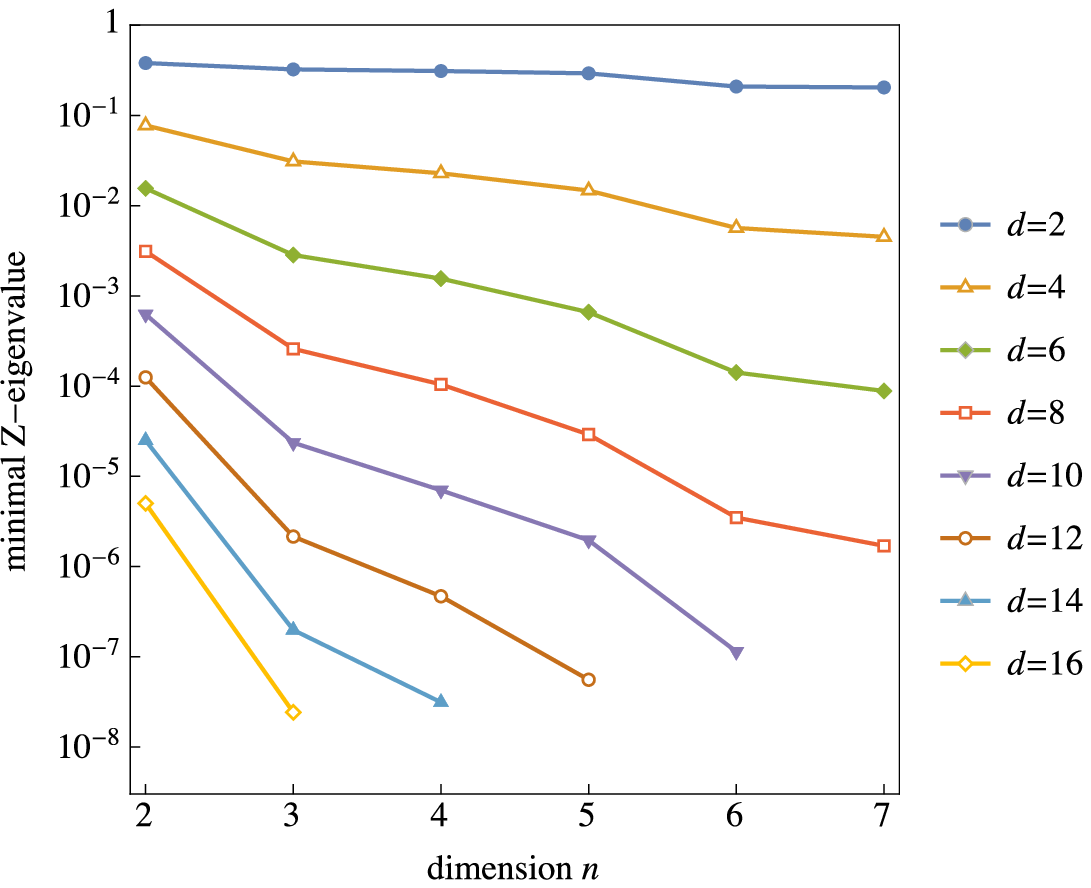}}}
\caption{Left: Minimal H-eigenvalues of the Smith tensor as functions of the dimension $n$. Right: Minimal Z-eigenvalues of the Smith tensor as functions of the dimension $n$.}\label{minimalHeigenvalues2}
\end{figure}
The threshold value $\tau=10$ was used for all cases. It is important to note that the GEAP algorithm does not guarantee convergence toward the minimal eigenvalue for each initial guess, and may instead converge to other local minima of~\eqref{geapopt}. This was found to occur in cases $d\geq 8$ and $n\geq 4$, $d\geq 14$ and $n\geq 3$, $d\geq 18$ and $n\geq 2$, however, the non-minimal eigenvalues returned by GEAP were always several orders of magnitude higher than the actual minimal eigenvalue. As a precaution, we utilized a prescreening stage where for each pair $(d,n)$, a set of $1\,000$ initial approximations of the minimal H- and Z-eigenvalues were computed, respectively, using $100$ iterations of TT-GEAP with different uniformly random initial guesses $x_0\in[-1,1]^n$. The initial guess producing the smallest eigenvalue in magnitude was then iterated using TT-GEAP until the termination criterion $|\lambda_k-\lambda_{k-1}|<10^{-14}$ was satisfied. The process terminated successfully for each pair $(d,n)$ and restarts were not necessary.

The minimal eigenvalues converge toward zero at an exponential rate for increasing order $d$ and fixed dimension $n$. The behavior of minimal H- and Z-eigenvalues is similar, both being monotonously decreasing for increasing dimension $n$. However, in the case of H-eigenvalues this becomes difficult to discern for increasing $d$. This behavior manifests itself as the overlapping H-eigenvalues for $n=3,4$ in the left-hand side of Figure~\ref{minimalHeigenvalues1} and as the plateau in the left-hand side of Figure~\ref{minimalHeigenvalues2}.

\subsection{Generalized eigenvalues using TT-GEAP}
The TT-GEAP algorithm also enables the computation of $B$-eigenpairs. We consider the minimal eigenvalues of the $B$-eigenvalue problem
\begin{equation}
Ax^{d-1}=\lambda Bx^{d-1},\label{lcmgcd}
\end{equation}
where $A=(S_d)$ and $B=[S_d]$ for $S=\{1,\ldots,n\}$. The problem~\eqref{lcmgcd} is a multidimensional extension of the generalized eigenvalue problem recently considered in~\cite{ik06}.

The results are presented in Figure~\ref{Bgraph1}, where we display
\begin{itemize}
\item[(v)] The absolute values of minimal $B$-eigenvalues for increasing order $d$ and fixed dimension $n$ (left-hand side).
\item[(vi)] The absolute values of minimal $B$-eigenvalues for increasing dimension $n$ and fixed order $d$ (right-hand side).
\end{itemize}
For $n=2$ and $(d,n)\in\{(4,3),(6,3)\}$, the threshold value $\tau=1$ was used. Otherwise the value $\tau=10$ was used instead. In this experiment, we utilized a prescreening stage where a set of $1\,000$ initial approximations of the minimal $B$-eigenvalues were computed using $100$ iterations of TT-GEAP with different uniformly random initial guesses $x_0\in[-1,1]^n$. The initial guess producing the smallest eigenvalue in magnitude was then iterated using TT-GEAP until the termination criterion $|\lambda_k-\lambda_{k-1}|<10^{-14}$ was satisfied. In the cases $n=2$ and $(d,n)=(6,5)$ the minimal $B$-eigenvalue has negative sign, in which case TT-GEAP has to be applied to the system
\[
Ax^{d-1}=\lambda(-B)x^{d-1}
\]
to obtain the actual minimal eigenvalue $-\lambda$. Due to the significantly more complex structure of the generalized eigenvalue spectrum, we only display the $B$-eigenvalues obtained using TT-GEAP which could be verified to be minimal by Mathematica 10's \texttt{NSolve} function.

By their behavior, the computed $B$-eigenvalues lie between H- and Z-eigenvalues: they tend to zero at a faster rate than respective H-eigenvalues, but slower than Z-eigenvalues. This should be contrasted with the matrix case, where the minimal $B$-eigenvalues of the system~\eqref{lcmgcd} tend to zero as $n\to\infty$ at a faster rate than the respective minimal eigenvalues of the Smith matrix.

The solution of dominant eigenpairs of~\eqref{lcmgcd} using TT-GEAP is possible for $n=2$. For $n>2$, however, we were not able to consistently recover the dominant $B$-eigenvalues and the experiments required several restarts in order to ensure convergence. The TT-GEAP algorithm appears to be very sensitive to the choice of the initial guess. Interestingly, in the low dimensional case it appears that the dominant $B$-eigenvalues are --- at least numerically --- nearly indistinguishable from saddle point solutions to~\eqref{geapopt}, making their solution very unstable using GEAP.

\begin{figure}[!t]
\centering
\subfloat{{\includegraphics[height=.433\textwidth]{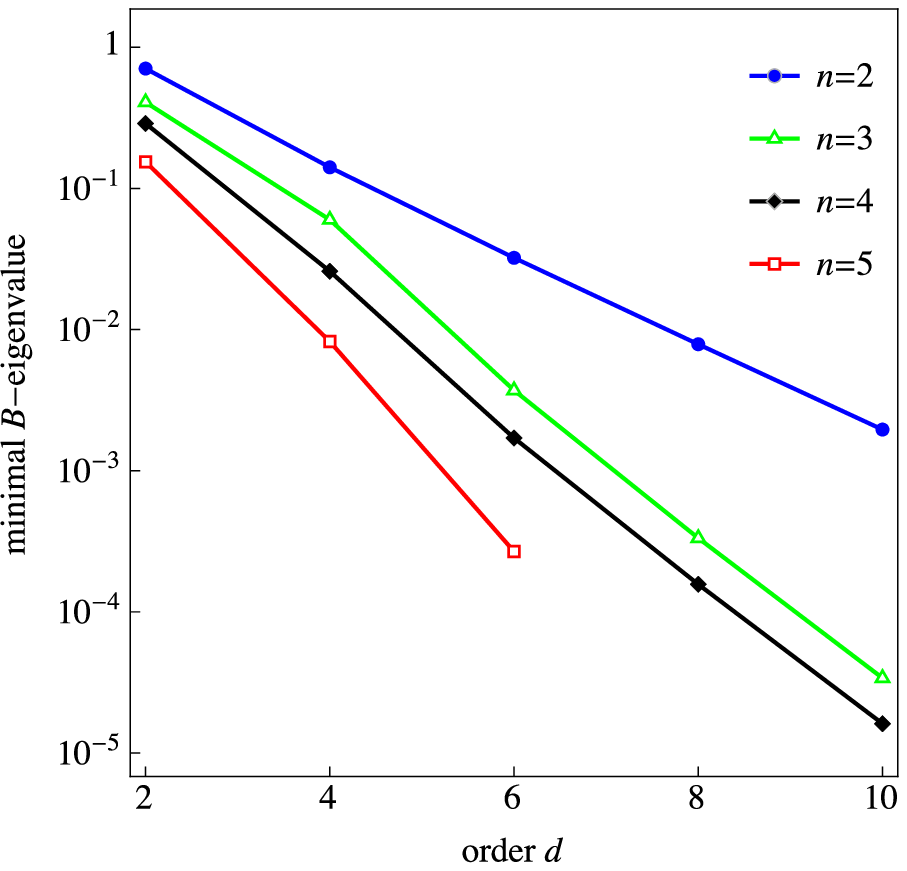}}}
\subfloat{{\includegraphics[height=.434\textwidth]{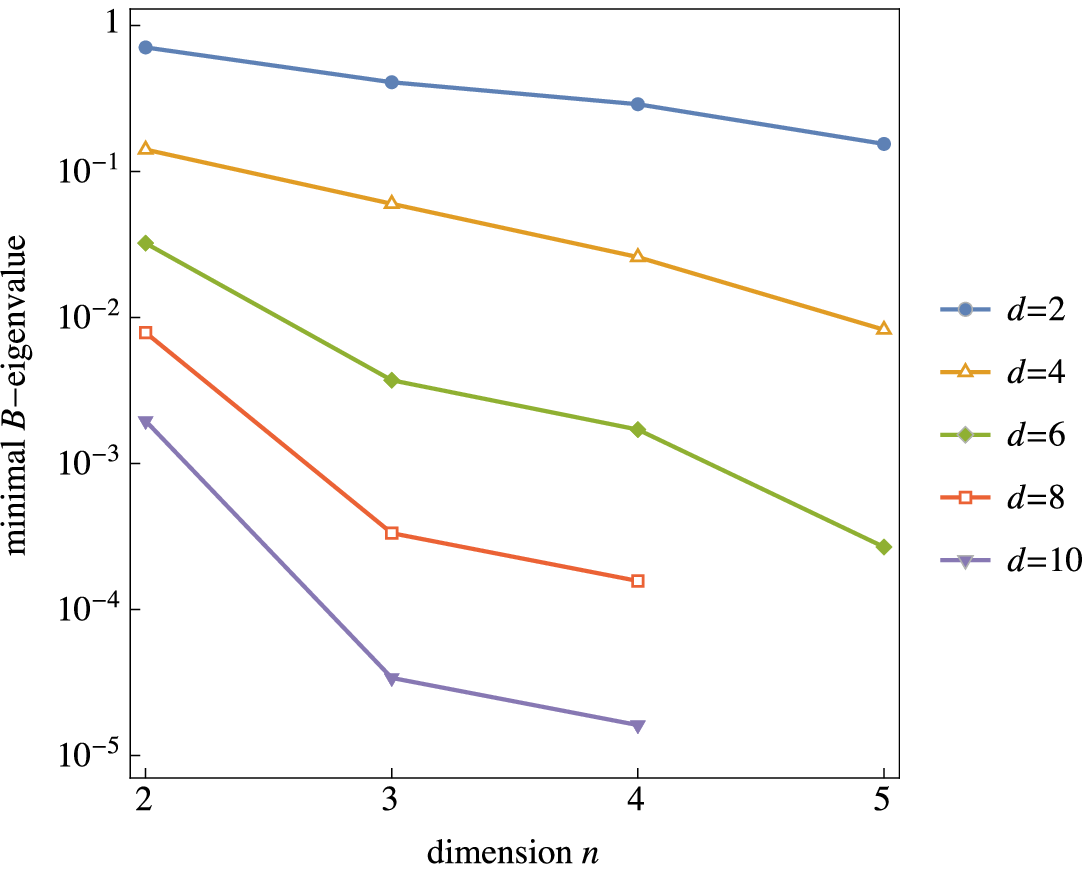}}}
\caption{Left: Absolute values of the minimal $B$-eigenvalues of the system~\eqref{lcmgcd} with increasing order $d$ and fixed dimension $n$. Right: Absolute values of the minimal $B$-eigenvalues with increasing dimension $n$ and fixed order $d$.}\label{Bgraph1}
\end{figure}

\section{Conclusions and future prospects}
Several algorithms for the solution of tensor eigenvalue problems have been proposed in recent years. In this work, we have modified the existing symmetric high-order power method (S-HOPM) and the generalized eigenproblem adaptive power method (GEAP) for use with the tensor-train format, which drastically lessens the computational complexity of these algorithms. The results of this paper are also applicable for low-rank tensors~\cite{GKT13} beyond the eigenvalues of the lattice-theoretic tensors.

We have demonstrated that meet tensors can be represented efficiently in the TT format with sparse TT cores and, in the special case of the Smith tensor, the tensor-train decomposition can be stored in only $\mathcal{O}(n\ln n)$ parameters independently of tensor order, making the numerical study of these objects feasible with very high dimensionality and order. For join tensors the TT-DMRG cross algorithm can be used to construct the TT decomposition numerically. Analogously to the case of the GCD tensors, we have shown that, in a special case of LCM tensors, the maximal TT rank grows in a controlled fashion, thus bounding the number of parameters needed to represent these tensors in the TT format. The TT-S-HOPM and TT-GEAP algorithms were used to compute the dominant and minimal H- and Z-eigenvalues of GCD tensors. We were also able to solve the minimal eigenvalues of the generalized eigenvalue problem for GCD tensors with respect to LCM tensors. 

The study of H- and Z-eigenvalues is still a largely underdeveloped field of study and theoretical bounds concerning these eigenvalues need to be researched further. For meet tensors an upper bound on eigenvalues has already been derived in a previous work~\cite{ilhyp} and this theoretical bound agrees well with our numerical experiments for GCD tensors. We have also considered the generalized eigenvalue problem for GCD and LCM tensors. This, we hope, motivates further study of generalized tensor eigenvalue problems. To this end, the decompositions of join tensors in particular require further attention.

The tensor-train decomposition has proven to be an invaluable tool in problems of high dimensionality and high order. Since a multitude of basic properties of tensor eigenvalue problems can be expressed naturally in terms of the tensor-train decomposition, it is a natural toolkit in the hands of the tensor eigenvalue community.

\section*{Acknowledgements}

The authors wish to thank the anonymous referees for their valuable comments and suggestions, which helped to improve this paper greatly.

\bibliographystyle{acm}
\bibliography{tevp}
\end{document}